\documentclass[11pt]{amsart}
\usepackage{amssymb}
\usepackage{amsmath}
\usepackage{amsthm}
\usepackage{amscd}
\usepackage{amsfonts}
\usepackage{tikz}
\usepackage{tikz-qtree}
\usepackage{tikz-cd}
\usepackage{mathrsfs}
\usepackage[all]{xy}
\usepackage{faktor}
\usepackage{graphicx}
\usepackage{blkarray}
\usepackage{amsmath}
\usepackage{lineno}
\usepackage{graphicx}
\usepackage{float}
\usepackage{tikz}
\usepackage{xparse}
\usepackage{ stmaryrd }
\usepackage{pgf}
\usepackage{enumerate}
\usetikzlibrary{calc,fit,matrix,arrows,automata,positioning}
\usetikzlibrary{arrows.meta}

\newtheorem{theorem}{Theorem}
\newtheorem{lemma}{Lemma}
\newtheorem{proposition}{Proposition}

\newtheorem{conjecture}{Conjecture}
\theoremstyle{definition}

\newtheorem{example}{Example}
\newtheorem{construction}{Construction}

\newtheorem{question}{Question}

\newcommand{\RN}[1]{
  \textup{\uppercase\expandafter{\romannumeral#1}}%
}
\newcommand{\RowNumber  }{\mathcal{R}}
\newcommand{\ColNumber  }{\mathcal{C}}
\usetikzlibrary{matrix,backgrounds}
\pgfdeclarelayer{myback}
\pgfsetlayers{myback,background,main}

\tikzset{mycolor/.style = {line width=1bp,color=#1}}%
\tikzset{myfillcolor/.style = {draw,fill=#1}}%

\newcommand{\tikzmark}[1]{\tikz[overlay,remember picture] \node (#1) {};}
\newcommand{\tikzdrawbox}[3][(0pt,0pt)]{%
    \tikz[overlay,remember picture]{
    \draw[#3]
      ($(left#2)+(-0.3em,0.9em) + #1$) rectangle
      ($(right#2)+(0.3em,-0.4em) - #1$);}
}

\NewDocumentCommand{\highlight}{O{blue!40} m m}{%
\draw[mycolor=#1] (#2.north west)rectangle (#3.south east);
}

\usepackage{chessfss}

\newlength{\symsize}
\newlength{\boardwidth}
\setboardfontsize{\symsize}

\newcommand{\board}[1]{\setlength{\fboxsep}{0pt}%
\fbox{\parbox{\boardwidth}{\setlength{\baselineskip}{\symsize}#1}}}
\newcommand{\row}[1]{\parbox[c][\symsize]{\symsize}{\hfill{#1}}}

\newcommand{\chessboard}[1]{\begin{tabular}{cc}
\parbox{\symsize}{\setlength{\baselineskip}{\symsize}
\row{} \row{} \row{} \row{} \row{} \row{} \row{} \row{} \row{} \row{} }&\board{#1}\\
\end{tabular}}

\newcommand{\TextOnWhite}[1]{\WhiteEmptySquare\hspace{-\symsize}%
\raisebox{.40\symsize}{\makebox[\symsize][c]{\small #1}}}
\newcommand{\TextOnBlack}[1]{\BlackEmptySquare\hspace{-\symsize}%
\raisebox{.40\symsize}{\makebox[\symsize][c]{\small #1}}}

\calclayout

\begin{document}

\title[Carlsson's rank conjecture and square-zero upper triangular matrices]{Carlsson's rank conjecture and a conjecture on square-zero upper triangular matrices}

\author{BERR\.{I}N \c SENT\" URK and  \" OZG\" UN \" UNL\" U }

\address{Department of Mathematics, Bilkent
University, Ankara, 06800, Turkey.}

\email{ berrin@fen.bilkent.edu.tr  \\
unluo@fen.bilkent.edu.tr }

\thanks{The second author is partially supported by T\"{U}BA-GEB\.{I}P/2013-22}



\begin{abstract}
Let $k$ be an algebraically closed field and $A$ the
polynomial algebra in $r$ variables with coefficients in $k$. In case the characteristic of $k$ is $2$, Carlsson
\cite{Carlsson1986} conjectured that for any $DG$-$A$-module $M$
of dimension $N$ as a free $A$-module, if the homology of $M$ is
nontrivial and finite dimensional as a $k$-vector space, then $2^r\leq N$.
Here we
state a stronger conjecture about varieties of square-zero upper triangular
$N\times N$ matrices with entries in $A$.
Using stratifications of these varieties via Borel orbits, we show that the
stronger conjecture holds when $N < 8$ or $r < 3$ without any restriction on the characteristic of $k$.
As a consequence, we obtain a new proof for many of the known cases of Carlsson's
conjecture
and give new results when $N > 4$ and $r = 2$.
\end{abstract}

\maketitle
\section{Introduction}
\label{section:Introduction}

A well-known conjecture in algebraic topology states
that if $(\mathbb{Z}/p\mathbb{Z})^r$ acts freely and cellularly on a finite
CW-complex homotopy equivalent to ${S^{n_1}\times \ldots \times S^{n_m}}$, then $r$ is
less than or equal to $m$. This conjecture is known to be true in several cases: In
the equidimensional case $n_1=\ldots=n_m=:n$, Carlsson \cite{Carlsson1982},
Browder \cite{Browder}, and Benson-Carlson
\cite{BensonCarlson} proved the conjecture under the assumption that the induced
action on homology is trivial. Without the homology assumption, the
equidimensional conjecture was proved by Conner \cite{Conner} for $m=2$,
Adem-Browder \cite{AdemBrowder} for $p\neq2$ or $n\neq 1,3,7$, and Yal\c{c}{\i}n \cite{EYgroupact} for $p=2$, $n=1$. In the
non-equidimensional case, the
conjecture is proved by Smith \cite{Smith} for $m=1$, Heller \cite{Heller} for
$m=2$, Carlsson \cite{Carlsson2} for $p=2$ and $m=3$ , Refai \cite{Refai1} for $p=2$ and $m=4$, and
Okutan-Yal\c{c}{\i}n \cite{OkutanYal} for products in which the average of the
dimensions
is sufficiently large compared to the differences between them. The general case
$m\geq 5$ is
still open.

Let $G=(\mathbb{Z}/p\mathbb{Z})^r$ and $k$ be an algebraically closed field of
characteristic $p$.
Assume that $G$ acts freely and cellularly
on a finite CW-complex $X$ homotopy equivalent to a  product of $m$ spheres. One can
consider the cellular
chain complex $C_*(X;k)$ as a finite chain complex of free $kG$-modules whose
homology  $H_*(X;k)$ is a $2^m$-dimensional $k$-vector space. Hence,
a stronger and purely algebraic conjecture can be stated as follows: If $C_*$ is a
finite chain complex of free $kG$-modules with nonzero homology then $\dim_k
H_*(C_*)\geq 2^r$. However, Iyengar-Walker in \cite{IyengarWalker}
disproved this algebraic conjecture when $p\neq 2$, but the algebraic version for $p=2$ remains open.

Let $R$ be a graded ring. A pair $(M,\partial)$ is a \emph{differential graded} $R$-module if $M$ is a graded $R$-module and $\partial$ is an $R$-linear endomorphism of $M$ that has degree $-1$ and satisfies $\partial^2=0$. Moreover, a $DG$-$R$-module is \emph{free} if the underlying $R$-module is free.

Let  $A=k[y_1,\ldots,y_r]$ be the polynomial
algebra in $r$ variables, where $k$ is a field and each $y_i$ has degree $-1$.
 Using a functor from the category
of chain complexes of $kG$-modules to the category of differential graded
$A$-modules, Carlsson showed in \cite{Carlssonbeta}, {\cite{Carlsson1986}} that
the above algebraic conjecture is  equivalent to the following conjecture when the characteristic of $k$ is $2$:
\begin{conjecture}\label{carlssonsconjecture}
Let $k$ be an algebraically closed field, $A$ the polynomial
algebra in $r$ variables with coefficients in $k$, and $N$ a positive integer. If
$(M,\partial)$ is a free $DG$-$A$-module of rank $N$ whose homology is nonzero
and finite dimensional as a $k$-vector space, then $N\geq 2^r$.
\end{conjecture}

 When the characteristic of $k$ is $2$, Conjecture \ref{carlssonsconjecture} was proved by Carlsson \cite{Carlsson2} for $r\leq 3$ and Refai \cite{Refai1}
for $N\leq 8$.  Avramov, Buchweitz, and Iyengar in \cite{Avramov}
dealt with regular rings and in particular they proved Conjecture \ref{carlssonsconjecture} for $r\leq 3$ without any restriction on the characteristic of $k$.
See also Proposition  $1.1$ and Corollary $1.2$ in \cite{AlldayPuppe2}, Theorem $5.3$ in \cite{WalkerARC} for results in characteristic not equal to $2$.

 In this paper we consider the conjecture from the viewpoint
of algebraic geometry. We show that Conjecture \ref{carlssonsconjecture}
is implied by the following in Section \ref{secconj1conj2}:

\begin{conjecture}\label{conj1}
Let $k$ be an algebraically closed field, $r$ a positive
integer, and  $N=2n$ an even positive integer. Assume that there exists a
nonconstant morphism $\psi $ from the projective variety $\mathbb{P}^{r-1}_k$ to the
weighted
quasi-projective variety of rank $n$ square-zero upper triangular $N\times N$
matrices $(x_{ij})$ with $\deg (x_{ij})=d_i-d_j+1$ for some
$N$-tuple of nonincreasing integers $(d_1,d_2,\dots ,d_N)$. Then $N\geq 2^r$.
\end{conjecture}

We will give a more precise statement of Conjecture \ref{conj1} in Section
\ref{sec3} after discussing necessary definitions and notation. We propose the
following, which is stronger than Conjecture \ref{conj1}:

\begin{conjecture}\label{conj2}
Let $k$, $r$, $N$, $n$ and $\psi $ be as in Conjecture \ref{conj1}. Assume $1\leq
\RowNumber  ,  \ColNumber  \leq N$ and the value of $x_{ij}$ at every point in the
image of $\psi $ is $0$ whenever $i\geq N-\RowNumber  +1$ or $j\leq \ColNumber  $.
Then $N\geq 2^{r-1}(\RowNumber  +\ColNumber )$.
\end{conjecture}

\noindent Note that in Conjecture \ref{conj2} we have $2^{r-1}(\RowNumber
+\ColNumber  )\geq 2^r$ because $\RowNumber\geq 1 $ and $\ColNumber\geq 1 $.
The main result of the paper is a proof of Conjecture \ref{conj2} when $N<8$
or $r<3$, see Theorem \ref{mainthm1} and Theorem \ref{mainthm2}. As Conjecture
\ref{conj2} is the strongest conjecture mentioned above, we obtain proofs of all the
conjectures in this introduction under the same conditions, including the
main result of Carlsson in \cite{Carlsson1986}. Also note that for $r=2$, taking $N>
4$
gives novel results not covered in the literature.
 However, when the characteristic of the field $k$ is not $2$, Iyengar-Walker \cite{IyengarWalker} gave a counterexample to Conjecture \ref{carlssonsconjecture} for each $r\geq8$. Hence by Section \ref{secconj1conj2}, we can say that Conjectures \ref{conj1} and \ref{conj2} are also false when $r\geq8$ and
the characteristic of the field $k$ is not $2$.  All these conjectures are still open in case the characteristic of $k$ is $2$. In Section \ref{Examples}, we conclude
with examples and problems.

We thank the referee for giving us extensive feedback, a shorter proof of Theorem \ref{THM1} and encouraging us to extend our results to fields with characteristics other than $2$. We are also grateful to Matthew Gelvin for comments and suggestions.
\section{Some notes on Conjectures \ref{carlssonsconjecture}, \ref{conj1}, and \ref{conj2}}
\label{secconj1conj2}
To show that Conjecture \ref{conj2} is the strongest conjecture in Section \ref{section:Introduction}, it is enough to prove the following theorem.
\begin{theorem}\label{THM1}
  Conjecture \ref{conj1} implies Conjecture \ref{carlssonsconjecture}.
\end{theorem}
\begin{proof}
 Let $k$, $r$, and $A$ be as in Conjecture \ref{carlssonsconjecture}. Let $(M,\partial)$ be a free $DG$-$A$-module of rank $N$ which satisfies the hypothesis in Conjecture \ref{carlssonsconjecture}. Without loss of generality, we can assume that $N$ is the smallest rank of all such $DG$-$A$-modules.

 Suppose the image of the differential $\partial$ is not contained in $(y_1,\ldots, y_r)M$. Then, there exists a basis $b_1,\ldots, b_N$ of $M$ and
  there are some $i$ and $j$ so that
$\partial(b_i) = c b_j + \sum_{l \ne j} g_l b_l$
for some non-zero $c \in k$ and some $g_l \in A$. Replacing $b_j$ with $\partial(b_i)$ gives a new
basis $b'_1, ..., b'_N$ such that $\partial(b'_i) = b'_j$. Now form the acyclic sub-$DG$-$A$-module
$(E,\partial)$ of $(M,\partial)$ spanned by $b'_i$, $b'_j$. The map $(M,\partial) \rightarrow (M,\partial)/(E,\partial)$ is a surjective
quasi-isomorphism and $M/E$ is free of rank $N-2$. This is a contradiction. Hence, the image of the differential $\partial$ is contained in $(y_1,\ldots, y_r)M$.

 Now pick any basis $c_1, \dots ,c_N$ of $M$
 such that $\deg (c_1) \leq \dots \leq \deg (c_N)$. Let m be such that $\deg (c_{N-m+1}) = \cdots = \deg (c_N)$ and
$\deg (c_i) < \deg (c_{N-m+1})$ for all ${i < N-m+1}$. For each $i$, we have $\partial(c_i) = \sum_j g_{ij}c_j $, for
some homogeneous polynomials $g_{ij}\in A$. Since the image of $\partial$ is contained in $(y_1,\ldots, y_r)M$, no $g_{ij}$ is a non-zero
constant. Thus, whenever $g_{ij}$ is non-zero, we have $\deg (g_{ij}) \leq -1$ and hence
${\deg (c_j) = \deg (c_i) - 1 - \deg (g_{ij}) \geq \deg (c_i)}$.
It follows that the differential on $M$ restricts to one on the submodule
\begin{displaymath}
L = Ac_{N-m+1} \oplus \cdots \oplus  Ac_N.
\end{displaymath}
 More precisely, for all $i\in\{N-m+1,\ldots, N\}$ we have
${\partial(c_i) = \sum_{j=N-m+1}^N g_{ij} c_j }$
where each nonzero $g_{ij}$ is a linear polynomial. Hence, relative to the
basis $c_{N-m+1}, \dots, c_N$, the differential $\partial$ on $L$ is given by a matrix in the form
$y_1 X_1 + \cdots + y_r X_r$
where each $X_i$ is an $m\times m$ matrix with entries in $k$. Since $\partial^2 = 0$ we have
$X_i^2 = 0$   and $X_i X_j + X_j X_i = 0$  for all i,j. In case the characteristic of the field $k$ is $2$, by a classical result
  about commuting set of matrices, for example see Theorem $7$ on page $207$ in \cite{HoffmanKunze}, there exists an invertible  $m \times m$ matrix $T$ with coefficients in $k$ such that $T^{-1}X_iT$ is upper triangular for all
$i\in \{1,\ldots, r\}$. In case the characteristic of the field $k$ is not $2$, for every polynomial $Q$ in noncommutative $r$ variables, the square of the matrix $Q(X_1,\ldots X_r)(X_iX_j-X_jX_i)$ is zero.
 Therefore, by a theorem of McCoy as stated in \cite{Shapiro}(see also \cite{Drazin}, \cite{McCoy}), again there exists a matrix $T$ as above which simultaneously conjugates all $X_i$'s to upper triangular matrices. In other words, there is a $k$-linear change of basis in which each
$X_i$ is upper triangular. It follows that, relative to this new basis $c'_{N-m+1}, \dots, c'_N$ of $L$
one has
$\partial(c'_i) \in \oplus_{j=i+1}^{N}  A c'_j$ for all  $i \in \{ N-m+1, \dots N\}$.
Note that $M/L$ is a free $DG$-$A$-module whose differential has an image in $(y_1,\ldots, y_r)(M/L)$ and so, by inductive arguments on rank, we may assume that $M/L$ admits a basis which makes its differential upper triangular. The union of any lift of this basis to $M$ with
the basis $c'_{N-m+1}, \dots, c'_N$ gives a basis $B$ for $M$ where $\partial$ is represented by an upper triangular matrix $\Psi'$.
 Moreover, Propositions $8$ and $9$
 in \cite{Carlsson2} work for any characteristic. Hence $N$ is divisible
by $2$ and for any $\gamma $ in $k^r-\{0\}$ the evaluation of $\Psi'$ at
$\gamma $ gives a matrix of rank $N/2$.

 Let $S=k[x_1,\dots ,x_r]$ be
the polynomial algebra with $\deg (x_i)=1$. For $1\leq i\leq r$, replace $y_i$ with $x_i$
 in $\Psi '$ to obtain $\Psi $. Note that $\Psi $ can be considered as a nonconstant morphism from the projective variety $\mathbb{P}^{r-1}_k$
to the
weighted quasi-projective variety of rank $N/2$ square-zero upper
triangular $N\times N$
matrices $(x_{ij})$ with $\deg (x_{ij})=d_i-d_j+1$ where $d_i= -($degree
of the $i$th element in $B)$.
\end{proof}

\section{Varieties of square-zero matrices}
\label{sec3}
 We assume that $k$ is an algebraically closed field, $n$ a positive integer, $N=2n$, and $d=(d_1,d_2,\dots ,d_N)$
  an $N$-tuple of nonincreasing integers.
\subsection{Statements of conjectures}
We give here the notation for the affine and projective varieties used
to prove the conjectures discussed above.
First we fix an affine variety $U_N$, a ring $R(U_N)$, and a
subvariety $V_N$ as follows:
\begin{itemize}
\item $U_N$ is the affine variety of strictly upper triangular $N\times N$ matrices
over $k$.
\item $R(U_N)=k[\,\,x_{ij}\,\,|\,\, 1\leq i< j\leq N\,\,]$ is the coordinate ring of
$U_N$.
\item $V_N$ is the subvariety of square zero matrices in $U_N$.
\end{itemize}
Define an action of the unit group $k^{*}$ on $U_N$ by $\lambda \cdot
(x_{ij})=(\lambda^{d_i-d_j+1}x_{ij})$ for $\lambda \in k^{*}$. Using this action we
set two more notation.
\begin{itemize}
\item $U(d)$ is the weighted projective space given by the quotient of $U_N-\{0\}$ by
the action of $k^{*}$.
\item  $R(U(d))$ is the homogeneous coordinate ring of $U(d)$. In other words,
$R(U(d))$ is
 $R(U_N)$ considered as a graded
ring with $\deg(x_{ij})=d_i-d_j+1$.
\end{itemize}
 Note that the polynomial  $\displaystyle p_{ij}=\sum_{m=i+1}^{j-1}x_{im}x_{mj}$ in
$R(U(d))$ is homogeneous of degree $d_i-d_j+2$ whenever $1\leq i< j\leq
N$. Similarly, the
$n\times n$-minors of $(x_{ij})$ are homogeneous polynomials in $R(U(d))$. Hence, we
define two subvarieties of $U(d)$ as follows:
\begin{itemize}
\item $V(d)$ is the projective variety of square zero matrices in $U(d)$.
\item $L(d)$ is the subvariety of matrices of rank less than $n$ in $V(d)$.
\end{itemize}
We can use this terminology to restate Conjecture \ref{conj1}:
\begin{conjecture}\label{conj1afternotations}
Let $k$ be an algebraically closed field, $r$ a positive
integer, and $d$ an $N$-tuple of nonincreasing integers. If there exists a
nonconstant morphism $\psi $ from the projective variety $\mathbb{P}^{r-1}_k$
to the quasi-projective variety $V(d)-L(d)$, then $N\geq 2^r$.
\end{conjecture}
\noindent Let $U$ be an open subset of $V(d)$. We say $\psi
:\mathbb{P}^{r-1}_k\rightarrow U$ is a nonconstant morphism if  $\psi $ is
represented by $(\psi_{ij})$ so that the following conditions are satisfied:
\begin{enumerate}[(I)]
\item there exists a positive integer $m$ so that each $\psi_{ij}$ is a homogeneous polynomial
in the variables $x_1, x_2, \ldots, x_r$ in $S$ of degree $m(d_i-d_j+1)$ for $1\leq i< j\leq N$,
\item for every $\gamma \in \mathbb{P}^{r-1}_k$ there exists $i,j$ such that
$\psi_{ij}(\gamma)\neq 0$.
\end{enumerate}
In particular, if
$\psi:\mathbb{P}^{r-1}_k\rightarrow U$ is a nonconstant morphism,  $\psi$ can be
considered as a function from
$\mathbb{P}^{r-1}_k$ to $U$ represented by a nonconstant polynomial map
$\widetilde{\psi }$ from $\mathbb{A}^{r}_k$ to the cone over $U$ such that
$\widetilde{\psi }(\mathbb{A}^{r}_k-\{0\})$ does not contain the zero matrix in
$V_N$. Each indeterminate $x_{ij}$ can be
viewed as homogeneous polynomial in $R(U(d))$. Hence for $1\leq \RowNumber
,\ColNumber  \leq N$
we define an important subvariety of $V(d)$:
\begin{itemize}
\item $V(d)_{\RowNumber  \ColNumber  }$ is the subvariety of $V(d)$ given by the
equations $x_{ij}=0$ for $i\geq N-\RowNumber  +1$ or $j\leq \ColNumber  $.
\end{itemize}
\noindent  Now we restate the Conjecture \ref{conj2} as follows:
\begin{conjecture}\label{conj2afternotations}
Let $k$ be an algebraically closed field, $r$ a positive
integer, and $d$ an $N$-tuple of nonincreasing integers. If there exists a
nonconstant morphism $\psi $ from the projective variety $\mathbb{P}^{r-1}_k$ to
 the quasi-projective variety ${V(d)_{\RowNumber  \ColNumber  }}-L(d)$,
  then $N\geq 2^{r-1}(\RowNumber  +\ColNumber  )$.
\end{conjecture}
  Hence, these varieties are the main interest in this paper.

\subsection{ Action of a Borel subgroup on $V_N$}
Here we introduce an action of a Borel subgroup in the group of invertible $N\times
N$ matrices on the varieties discussed in the previous subsection. First we set a
notation for the Borel subgroup.
\begin{itemize}
\item $B_{N}$ is the group of invertible upper triangular
$N\times N$ matrices with coefficients in $k$.
\end{itemize}
The group $B_{N}$  acts on $V_{N}$ by conjugation.
\begin{itemize}
\item $V_N/B_{N}$ denotes the set of orbits of the action of $B_{N}$ on $V_N$.
\item $B_{X}$ denotes the $B_N$-orbit that contains $X\in V_N$.
\end{itemize}
A \emph{partial permutation matrix} is a matrix having at most one nonzero entry, which is $1$,
in each row and column. A result of
Rothbach (Theorem $1$ in \cite{Rothbach}) implies that each $B_N$-orbit of $V_N$ contains a unique partial permutation
matrix.
Hence we introduce the following notation:
\begin{itemize}
\item $\mathbf{PM}(N)$ denotes the set of nonzero $N\times N$ strictly upper
triangular
square-zero partial permutation matrices.
\end{itemize}
There is a one-to-one correspondence between $\mathbf{PM}(N)$ and $V_N/B_{N}$
sending $P$ to
$B_{P}$. We can identify these partial permutation matrices with a subset of
 the symmetric group $Sym(N)$:
\begin{itemize}
\item $\mathbf{P}(N)$ is the set of involutions in $Sym(N)$; i.e., the set of
non-identity permutations whose square is the identity $()$.
\end{itemize} For $P\in \mathbf{PM}(N)$ and $\sigma \in \mathbf{P}(N)$,
\begin{itemize}
\item $\sigma _P$ denotes the permutation in $\mathbf{P}(N)$ that sends $i$ to $j$
if $P_{ij}=1$;
\item $P _{\sigma }$ denotes the partial permutation matrix in $\mathbf{PM}(N)$ that
satisfies $(P_{\sigma})_{ij}=1$ if and only if $\sigma (i)=j$ and $i<j$.
\end{itemize}
Clearly the assignments $P\mapsto \sigma _P$ and $\sigma \mapsto P _{\sigma }$ are
mutual inverses and so define a one-to-one correspondence between $\mathbf{P}(N)$
and $\mathbf{PM}(N)$.

\subsection{A partial order on the set of orbits}
There are important partial orders on $V_N/B_{N}$, $\mathbf{P}(N)$,
$\mathbf{PM}(N)$, all of which are
equivalent under the one-to-one correspondence mentioned above
(cf. \cite{Rothbach}).
We begin with $V_N/B_{N}$. For Borel orbits $B, B'\in V_N/B_{N}$,
\begin{itemize}
\item  $B'\leq B$ means the closure of $B$, considered as a subspace of $V_N$,
contains $B'$.
\end{itemize}
Second, we define a partial order on $\mathbf{PM}(N)$.
To do this, we consider ranks of certain minors of partial permutation matrices. In
general, for an $N\times N$ matrix $X$,
\begin{itemize}
\item $r_{ij}(X)$ denotes the rank of the lower left $((N-i+1)\times j)$ submatrix
of $X$, where $1\leq i < j\leq N$.
\end{itemize}
For partial permutation matrices $P', P\in \mathbf{PM}(N)$,
\begin{itemize}
\item  $P'\leq P$ means $r_{ij}(P')\leq r_{ij}(P)$ for all $i,j$.
\end{itemize}
Third, we define a partial order on $\mathbf{P}(N)$. For  positive integers
$i<j$, let $\sigma (i,j)$ denote the
product of the permutations $\sigma $ and $(i,j)$ and $\sigma^{(i,j)}$ the conjugate
of $\sigma$ by $(i,j)$. For $\sigma,\sigma ' \in \mathbf{P}(N)$,
\begin{itemize}
\item $\sigma '\leq \sigma$ if $\sigma '$ can be
obtained from $\sigma$ by a sequence of moves of the following form:
\begin{itemize}
\item Type {\RN{1}} replaces $\sigma$ with $\sigma (i,j)$ if $\sigma(i)=j$
and $i\neq j$.
\item Type {\RN{2}} replaces $\sigma$ with $\sigma^{(i,i')}$ if
$\sigma(i)=i<
i' < \sigma(i')$.
\item Type {\RN{3}} replaces $\sigma$ with $\sigma^{(j,j')}$ if
$\sigma(j)<\sigma(j')<j'<j$.
\item Type {\RN{4}} replaces $\sigma$ with $\sigma^{(j,j')}$ if $
\sigma(j')< j'< j=\sigma(j)$.
\item Type {\RN{5}} replaces $\sigma$ with $\sigma^{(i,j)}$ if
$i<\sigma(i)<\sigma(j)<j$.
\end{itemize}
\end{itemize}
The idea of describing order via these moves comes from \cite{Oliver}. Although we use
different names for moves, the set of possible moves are same. We represent a
permutation $(i_1,j_1)(i_2,j_2)\ldots(i_s,j_s)$ in $\mathbf{P}(N)$ by the matrix
\begin{displaymath}
{\left(\begin{array}{cccc}
  i_1& i_2&\ldots&i_s  \\j_1 & j_2&\ldots&j_s
\end{array}\right)}.
\end{displaymath}

For example, we draw the Hasse diagram of $\mathbf{P}(4)$
in which each edge is labelled by the type of the move it represents:

\begin{figure}[H]
    \centering
    \resizebox{0.8\textwidth}{!}{%
\begin{tikzpicture}[scale=.8]
  \node (one) at (-6,8) {$\tiny{\left(\begin{array}{cc}
  1 & 3  \\2 & 4 \end{array}\right)}$};
  \node (two) at (6,8) {$\tiny{\left(\begin{array}{cc}
  1 & 2  \\4 & 3 \end{array}\right)}$};
  \node (three) at (-9,3) {$\tiny{\left(\begin{array}{c}
  3  \\ 4 \end{array}\right)}$};
  \node (four) at (-3,3) {$\tiny{\left(\begin{array}{c}
  2  \\ 3 \end{array}\right)}$};
  \node (five) at (3,3) {$\tiny{\left(\begin{array}{cc}
  1 & 2 \\3 & 4 \end{array}\right)}$};
  \node (six) at (9,3) {$\tiny{\left(\begin{array}{c}
  1  \\ 2 \end{array}\right)}$};
  \node (seven) at (-6,0) {$\tiny{\left(\begin{array}{c}
  2  \\ 4 \end{array}\right)}$};
  \node (eight) at (6,0) {$\tiny{\left(\begin{array}{c}
  1  \\ 3 \end{array}\right)}$};
  \node (nine) at (0,-3) {$\tiny{\left(\begin{array}{c}
  1 \\ 4 \end{array}\right)}$};
  \draw (one) to node {$\begin{array}{cc}
 {\RN{1}} & \\ &  \end{array}$} (three) to node {$\begin{array}{cc}
 & \\ {\RN{2}}&  \end{array}$} (seven) to node {$\begin{array}{cc}
 & \\ {\RN{2}}&  \end{array}$} (nine);
  \draw (one) to node[below] {$\begin{array}{cccc}
    & {\RN{5}} & &\\
  & &       & \\
  & &        &\\
  & &     & \\
  & &     & \\
  & &     & \\
  \end{array}$} (five) to node {$\begin{array}{cc}
 & {\RN{1}}\\ &  \end{array}$} (eight) to node {$\begin{array}{cc}
 & \\ & {\RN{4}} \end{array}$} (nine);
  \draw (two)  to node[above] {$\begin{array}{ccccccc}
&& &&&&\\ &&&&&&\\&&&&&&\\ &&&&&&\\&&&&&&\\&&&&&& \\&&&&&&\\{\RN{1}}&&&&&& \\&&&&&&
\end{array}$} (four)  to node[above] {$\begin{array}{cc}
& \\ &\\& \\ &\\&\\&\\{\RN{4}}& \\   \end{array}$} (seven);
  \draw (four) to node[below] {$\begin{array}{cc}
&{\RN{2}}\\ &    \end{array}$} (eight);
  \draw (five) to node[below] {$\begin{array}{cc}
{\RN{1}}&\\& \\   \end{array}$} (seven);
  \draw (two) to node[above] {$\begin{array}{ccccc}
&&&& \\ &&&&\\&&&& \\ &&&&\\&&&&\\& &&&{\RN{3}}\\   \end{array}$} (five);
  \draw (one) to node[above]
  {$\begin{array}{cccc}
&&&\\ &&&\\&&& \\ && &\\ &&&{\RN{1}}\\&&&\\ \end{array}$}
 (six) to node[below] {$\begin{array}{cc}
&{\RN{4}} \\ &\\& \\   \end{array}$} (eight);
  \end{tikzpicture}
    }%
    \caption{ Hasse diagram of $\mathbf{P}(4)$}
    \label{fig:1}
  \end{figure}
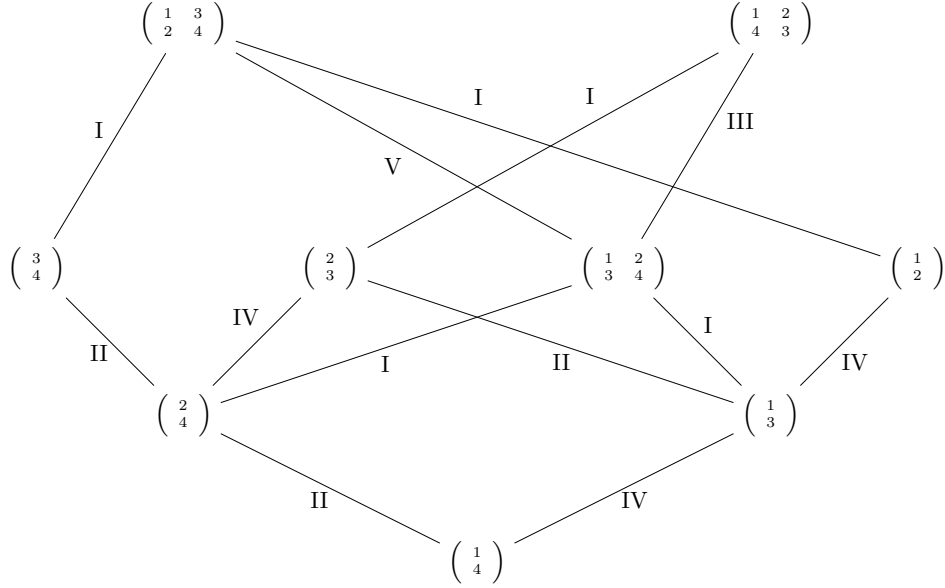
When $N\geq 6$, the Hasse diagram for $\mathbf{P}(N)$ is too large to draw here. We
are actually
only interested in a small part of this diagram, which we discuss in Section
\ref{rankoforbitsandfirstresult}.

One can consider Figure \ref{fig:1} as a stratification of $V_4$. In the next
section, we use the stratification of $V_N$ to stratify $V(d)$.
\subsection{Stratification of $V(d)$}
For $d=(d_1,d_2,\dots ,d_N)$ an $N$-tuple of
nonincreasing integers, $\lambda \in k^{*}$, and
$X=(x_{ij})\in V_{N}$, we have
\begin{displaymath}
\lambda \cdot X =\lambda\cdot (x_{ij})=
(\lambda^{d_i-d_j+1}x_{ij})=
D_{\lambda}I_{\lambda }XD_{\lambda }^{-1}
\end{displaymath}
where $D_{\lambda }$ denotes the diagonal matrix with entries
$\lambda ^{d_1}, \lambda ^{d_2},\dots , \lambda ^{d_N}$ and
$I_{\lambda }$ is the scalar matrix with
all diagonal entries $\lambda $. Let $P_X\in \mathbf{PM}(N)$ be
the unique partial permutation matrix in the Borel orbit of $X$.
 Consider
$b\in B_N$ such that $$P_X=b^{-1}X b.$$
Let $I_{\lambda,X}$ be the diagonal matrix whose
$j\textsuperscript{th}$ entry is $\lambda $ if $(P_X)_{ij}=1$ for some $i$ and
$1$ otherwise. Then we have
$$I_{\lambda }\,P_X \,=\,I_{\lambda,X}^{-1}\,P_X\,I_{\lambda,X}.$$
Hence, we have
$$\lambda \cdot X \,=\,D_{\lambda}\,
b\,I_{\lambda,X}^{-1}\,b^{-1}\,X\,b\,I_{\lambda,X}\,b^{-1}\,D_{\lambda
}^{-1}=Z^{-1}XZ$$
where $Z=b\,I_{\lambda,X}\,b^{-1}\,D_{\lambda }^{-1}$ is in $B_N$. Thus, for any
$X\in V(d)$ there
exists a well-defined Borel orbit in $V_N/B_N$ that contains a representative of
$X$ in $V_N$. Hence we can set the following notation. For $X\in V(d)$,
\begin{itemize}
\item $B_{X}$ denotes the Borel orbit in $V_N/B_N$ that contains a representative
of $X$ in  $V_N$.
\end{itemize}

Let $\psi:\mathbb{P}^{r-1}_k \rightarrow V(d)-L(d)$ be a nonconstant morphism.
 There is a lift of this morphism to a
morphism from $\mathbb{A}^{r}_k-\{0\}$ to the cone over $V(d)-L(d)$ that can
be extended to a morphism $\widetilde{\psi }:\mathbb{A}^{r}_k\rightarrow V_N$.
Since  $\mathbb{A}^{r}_k$ is an irreducible affine variety, there exists a unique
maximal Borel orbit among the Borel orbits that intersects the
image of $\widetilde{\psi } $
nontrivially. Note that this maximal Borel orbit is independent of
the lift and extension we selected because it is also maximal in the set $\{\,
B_X\, |\, X\in V(d) \}$. Hence we may associate a permutation to the nonconstant
morphism $\psi $:
\begin{itemize}
\item $\sigma _{\psi }$ is the permutation that corresponds to the unique maximal
Borel
orbit $B_X$ where $X$ is in the image of $\psi $.
\end{itemize}
Note that every point in the image of a morphism $\psi $ as above must have rank
$n$. Hence $\sigma _{\psi }$ must be a product of $n$ distinct transpositions. In
Section \ref{rankoforbitsandfirstresult}, we will restrict our attention to such
permutations.
\subsection{Operations on polynomial maps from $\mathbb{A}^{r}_k$ to
$V_N$}\label{opsection}
Another way to see that $B_X$ is well-defined for $X\in V(d)$ is to consider the
fact that a minor of a representative of $X$ is zero if and only if the
corresponding minor of another representative is zero.
We use this fact several times to prove our main result.
Hence we introduce the following notation. For
$X\in V_N$,
\begin{itemize}
\item $m^{i_1\,i_2\,\dots \,i_k}_{j_1\, j_2\, \dots \,j_k}(X)$ denotes the
determinant of
the $k\times k$ submatrix obtained by taking the $i_1\textsuperscript{th}$,
$i_2\textsuperscript{th},\dots ,{i_k}\textsuperscript{th}$ rows and
${j_1}\textsuperscript{th},{j_2}\textsuperscript{th}, \dots
,{j_k}\textsuperscript{th}$ columns of $X$.
\end{itemize}
First note that $m^{i_1\,i_2\,\dots \,i_k}_{j_1\, j_2\, \dots \,j_k}$ can be
considered as a morphism from $V_N$ to $k$, and hence can be composed with the
morphism
$\widetilde{\psi }$ mentioned in the previous subsection. Here we discuss several other
morphisms that we can compose with such morphisms.
For $u\in k$,
\begin{itemize}
\item $R_{i,j}(u)$ is the function that takes a square matrix $M$
and multiplies the $i\textsuperscript{th}$ row of
$M$ by $u$ and adds it to the $j\textsuperscript{th}$ row of $M$ while multiplying
the $j\textsuperscript{th}$ column of $M$ by $u$ and adding it to the
$i\textsuperscript{th}$ column of $M$.
\end{itemize}
Note that $R_{i,j}(u)(M)$ is a conjugate of $M$. In fact, they are in the same Borel
orbit when $M\in V_N$ and $i>j$. Hence, for $i>j$, we can consider $R_{i,j}(u)$
as an operation that takes a morphism from $\mathbb{A}^{s}_k$ to $V_N$
and transforms it to a morphism from $\mathbb{A}^{s+1}_k$ to $V_N$ by considering $u$
as a new indeterminate and applying $R_{i,j}(u)$ to the morphism. For $v\in k^{*}$,
\begin{itemize}
\item $D_i(v)$ denotes the function that takes a square matrix $M$ and
multiplies the $i\textsuperscript{th}$ row of $M$ by $v$ and the
$i\textsuperscript{th}$
column of $M$ by $1/v$.
\end{itemize}
Let $q$ be a polynomial in $s$ indeterminates. We define $D_i(q)$ as an
operation that takes a rational map from the quasi-affine variety
$\mathbb{A}^{s}_k-Z$ to $V_N$
and transforms it into a rational map from $\mathbb{A}^{s}_k-Z\cup V(q)$ to $V_N$ by
applying $D_i(q)$, using the following notation:
\begin{itemize}
\item $V(q_1,q_2,\dots ,q_k)$ is the variety determined by
the equations $q_1=q_2=\dots q_k=0$.
\end{itemize}
We use the above notation also for varieties in projective spaces determined by the
homogeneous polynomials $q_1,q_2,\dots ,q_k$.

\subsection{Rank of orbits and proof of first main result}
\label{rankoforbitsandfirstresult}
Each $\sigma \in \mathbf{P}(N)$ is a product of disjoint transpositions. Hence for
${\sigma \in \mathbf{P}(N)}$,
 we define the \emph{rank} of $\sigma$ to be the number of transpositions in $\sigma$.
Note that under the one-to-one corespondence between $\mathbf{P}(N)$ and
$\mathbf{PM}(N)$, the rank of
a permutation is equal to the rank of the corresponding partial permutation matrix.
\begin{itemize}
\item  $\mathbf{RP}(N)$ denotes the permutations in $\mathbf{P}(N)$ of rank
$n$.
\end{itemize}
Note that all moves other than type {\RN{1}} preserve  the rank of $\sigma$. Indeed,
the only way of obtaining $\sigma$ of smaller rank by applying our moves is by
deleting a
transposition, which is the effect of move of type {\RN{1}}. Also note that
it is impossible to have a move of type {\RN{2}} or a move of type {\RN{4}} between
two
permutations in $\mathbf{RP}(N)$. For example, we draw the Hasse diagram for
$\mathbf{RP}(6)$ where each dotted line denotes a move of type \RN{3} and solid line
denotes a move of type \RN{5}:
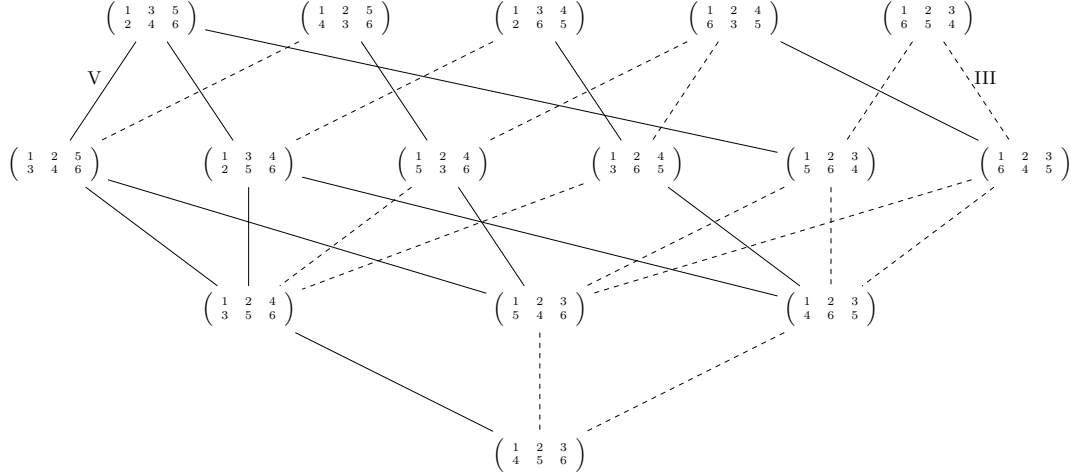
\begin{figure}[H]
    \centering
    \resizebox{0.9\textwidth}{!}{%
\begin{tikzpicture}
  \node (one) at (-8,3) {$\tiny{\left(
\begin{array}{ccc}
  1 & 3 & 5 \\2 & 4 & 6 \end{array}\right)}$};
  \node (two) at (-4,3) {$\tiny{\left(\begin{array}{ccc}
  1 & 2 & 5 \\4 & 3 & 6 \end{array}\right)}$};
  \node (three) at (0,3) {$\tiny{\left(
\begin{array}{ccc}
  1 & 3 & 4 \\ 2 & 6 & 5 \end{array}\right)}$};
  \node (four) at (4,3) {$\tiny{\left(\begin{array}{ccc}
  1 & 2 & 4 \\6 & 3 & 5 \end{array}\right)}$};
  \node (five) at (8,3) {$\tiny{\left(\begin{array}{ccc}
  1 & 2 & 3 \\6 & 5 & 4 \end{array}\right)}$};
  \node (six) at (-10,0) {$\tiny {\left(
\begin{array}{ccc}
  1 & 2 & 5 \\3 & 4 & 6 \end{array}\right)}$};
  \node (seven) at (-6,0) {$\tiny{\left(\begin{array}{ccc}
  1 & 3 & 4 \\2 & 5 & 6 \end{array}\right)}$};
  \node (eight) at (-2,0) {$\tiny {\left(\begin{array}{ccc}
  1 & 2 & 4 \\5 & 3 & 6 \end{array}\right)}$};
  \node (nine) at (2,0) {$ \tiny{\left(\begin{array}{ccc}
  1 & 2 & 4 \\3 & 6 & 5 \end{array}\right)}$};
  \node (ten) at (6,0) {$ \tiny{\left(\begin{array}{ccc}
  1 & 2& 3 \\5 & 6 & 4 \end{array}\right)}$};
  \node (eleven) at (10,0) {$ \tiny {\left(\begin{array}{ccc}
  1 & 2 & 3 \\6 & 4 & 5 \end{array}\right)}$};
  \node (twelve) at (-6,-3) {$\tiny {\left(\begin{array}{ccc}
  1 & 2 & 4 \\3 & 5 & 6 \end{array}\right)}$};
  \node (thirteen) at (0,-3) {$\tiny {\left(\begin{array}{ccc}
  1 & 2 & 3 \\5 & 4 & 6 \end{array}\right)}$};
  \node (fourteen) at (6,-3) {$\tiny {\left(\begin{array}{ccc}
  1 & 2 & 3 \\4 & 6 & 5 \end{array}\right)}$};
  \node (fifteen) at (0,-6) {$\tiny {\left(\begin{array}{ccc}
  1 & 2 & 3 \\4 & 5 & 6 \end{array}\right)}$};
  \draw (fifteen) -- (twelve);
  \draw (one) to node {$\begin{array}{cc}
 {\RN{5}} & \\ &  \end{array}$} (six);
  \draw (twelve) -- (seven) -- (one);
  \draw (one) -- (ten);
  \draw[dashed] (ten) -- (fourteen) -- (fifteen);
  \draw[dashed] (two) -- (six);
  \draw (six) -- (thirteen);
  \draw (six) -- (twelve);
  \draw (two) -- (eight) -- (thirteen);
  \draw[dashed] (three) -- (seven);
  \draw (three) -- (nine) -- (fourteen);
  \draw[dashed] (four) -- (eight) -- (twelve);
  \draw[dashed] (four) -- (nine) -- (twelve);
  \draw[dashed] (five) -- (ten) -- (thirteen) -- (fifteen);
  \draw (four) -- (eleven);
  \draw (seven) -- (fourteen);
  \draw[dashed] (eleven) -- (thirteen);
  \draw[dashed] (five) to node {$\begin{array}{cc} & {\RN{3}}\\ &  \end{array}$}
(eleven) -- (fourteen);
\end{tikzpicture}
 }%
    \caption{Hasse diagram of $\mathbf{RP}(6)$}
    \label{fig:2}
  \end{figure}
Such Hasse diagrams, with particular attention paid to the maximal elements, will lead to the proof of our
first main result.
\begin{theorem} \label{mainthm1}
Conjecture \ref{conj2afternotations} holds for $N<8$.
\end{theorem}
\begin{proof}
Take $N<8$, $d=(d_1,d_2,\dots ,d_N)$ an $N$-tuple of nonincreasing integers,
and $\psi: \mathbb{P}^{r-1}_k\rightarrow V(d)-L(d)$ a nonconstant morphism.
 Then $\sigma =\sigma _{\psi }$ is in
$\mathbf{RP}(N)$. By considering Figures \ref{fig:1} and \ref{fig:2}, we note that there exists
a unique maximal $\sigma ' \in \mathbf{RP}(N)$ such that $\sigma $ can be
obtained from $\sigma '$
by a sequence of moves of type \RN{3}.
Since moves of type \RN{3} do not change the number of leading zero rows and ending
zero columns of the corresponding partial permutation matrices,
the Borel orbit corresponding to $\sigma$ is contained in
$V(d)_{\mathcal{R}\mathcal{C}}$
if and only if the Borel orbit corresponding to $\sigma '$ is
 contained in $V(d)_{\mathcal{R}\mathcal{C}}$ for all $\mathcal{R}, \mathcal{C}$.
 Hence it is enough to consider the cases where $\sigma$ is
less than or equal to a maximal element in $\mathbf{RP}(N)$ for $N=2,4,6$.
We cover these cases by proving in the following eight statements: \\
{\bf (i)} If $\sigma =(1,2)$ then $r<2$.  \\
Assume to the contrary that $\sigma =(1,2)$ and $r\geq 2$. If we also write $\psi$ for its restriction to $\mathbb{P}^{1}_k\subseteq \mathbb{P}^{r-1}_k$, we get a map of the form
$$\psi (x:y)=\left[ \begin{array}{cc}
               0 & p_{12} \\
               0 &  0
             \end{array}  \right],
$$
where $p_{12}$ is a homogeneous polynomial in $k[x,y]$. Since $k$ is algebraically
closed,
 there exists $\gamma \in \mathbb{P}^{1}_k$ such that $p_{12}(\gamma)=0$. This means
$\psi(\gamma)$ is in
$L(d)$, which is a contradiction. \\
{\bf (ii) } If $\sigma \leq (1,2)(3,4)$ then $r<3$. \\
Suppose to the contrary that $\sigma \leq (1,2)(3,4)$ and $r\geq3$.
When we restrict $\psi $ to $\mathbb{P}^{2}_k$, we get a map of the form
$$\psi (x:y:z)=\left[ \begin{array}{cccc}
                      0 & p_{12} & p_{13} & p_{14} \\
                      0 & 0 & 0 & p_{24}  \\
                      0 & 0 & 0 & p_{34} \\
                      0 & 0 & 0 & 0
                    \end{array}
 \right].
$$
Note that there exists $\gamma $ in $\mathbb{P}^{2}_k$ such that
\begin{displaymath}
{p_{12}}( \gamma )= 0 \mbox{ and }p_{13}(\gamma)=0.
\end{displaymath}
Again this means $\psi(\gamma)\in L(d)$. Hence this case is proved by
contradiction as well. \\
{\bf (iii)} If $\sigma \leq (1,4)(2,3)$ then $r<2$. \\
Suppose we have
$$\psi (x:y)=\left[ \begin{array}{cccc}
                      0 & 0 & p_{13} & p_{14} \\
                      0 & 0 & p_{23} & p_{24}  \\
                      0 & 0 & 0 & 0 \\
                      0 & 0 & 0 & 0
                    \end{array}
 \right].
$$
Let $m^{i_1\,i_2\,\dots \,i_k}_{j_1\, j_2\, \dots \,j_k}$ be as in Section
\ref{opsection}
and use the same notation to denote its
composition with $\psi$.
Then there exists $\gamma $ in $\mathbb{P}^{1}_k$ such that
\begin{displaymath}
m^{12}_{34}(\gamma )=(p_{13}p_{24}-p_{23}p_{14})(\gamma)=0.
\end{displaymath}
This again gives a contradiction. \\
{\bf (iv)}  If $\sigma \leq (1,2)(3,4)(5,6)$ then $r<3$. \\
Suppose otherwise. We have
$$\psi (x:y:z)=\left[ \begin{array}{cccccc}
                      0 & p_{12}& p_{13} & p_{14} & p_{15} & p_{16}\\
                      0 & 0     & 0      & p_{24} & p_{25} & p_{26}\\
                      0 & 0     & 0      & p_{34} & p_{35} & p_{36}\\
                      0 & 0     & 0      & 0      & 0      & p_{46}\\
                      0 & 0     & 0      & 0      & 0      & p_{56} \\
                      0 & 0     & 0      & 0      & 0      & 0
                    \end{array}
 \right].
$$
If $p_{12}$ and $p_{13}$ are not relatively prime homogeneous polynomials then
there exists $\gamma \in \mathbb{P}^{2}_k$ such that
$$p_{12}(\gamma )=0\mbox{, }p_{13}(\gamma )=0\mbox{, and }  m^{123}_{456}(\gamma) =0.$$
 Moreover, if $p_{46}(\gamma)=0$ and $p_{56}(\gamma)=0$,
then the rank of $\psi(\gamma)$ is at most $2$, which leads to a contradiction. Hence we have
$p_{46}(\gamma)\neq 0$ or $p_{56}(\gamma)\neq 0$. Let
\begin{displaymath}
c_4(\gamma):=
\left[
\begin{array}{c}
p_{14}(\gamma) \\
p_{24}(\gamma) \\
p_{34}(\gamma)
 \end{array}
 \right]
 \mbox{ and }
 c_5(\gamma):=
\left[
\begin{array}{c}
p_{15}(\gamma) \\
p_{25}(\gamma) \\
p_{35}(\gamma)
 \end{array}
 \right].
\end{displaymath}
Since ${\psi}^2=0$,
$c_4(\gamma)p_{46}(\gamma)+c_5(\gamma)p_{56}(\gamma)=0$. By the fact that $p_{46}(\gamma)\neq 0$ or $p_{56}(\gamma)\neq 0$,
 $c_4(\gamma)$ and $c_5(\gamma)$
 are linearly dependent. Thus the rank of $\psi(\gamma)$ is at most $2$, which is a contradiction.\\
 Therefore we may assume $p_{12}$ and $p_{13}$ are
relatively prime. Since $\psi^2=0$, we have
$p_{12}p_{24}+p_{13}p_{34}=0$ and $p_{12}p_{25}+p_{13}p_{35}=0$. This implies that $p_{12}$
divides $p_{34}$ and $p_{35}$, and similarly $p_{13}$ divides $p_{24}$ and $p_{25}$.
Then there exists $\gamma $ in $\mathbb{P}^{2}_k$ such that
$$p_{12}(\gamma)=0 \mbox{ and }p_{13}(\gamma)=0.$$
This means $p_{12}$, $p_{13}$, $p_{24}$, $p_{25}$, $p_{34}$, and $p_{35}$
all vanish at $\gamma $. Hence $\psi(\gamma )\in L(d)$, which is a
contradiction.\\
{\bf (v)}  If $\sigma \leq (1,2)(3,6)(4,5)$ then $r<3$, and
{\bf (vi)} If $\sigma \leq (1,4)(2,3)(5,6)$ then $r<3$. \\
These cases are symmetric, so it is enough to prove {\bf{(v)}}.
Consider
$$\psi (x:y:z)=\left[ \begin{array}{cccccc}
                      0 & p_{12} & p_{13} & p_{14} & p_{15} & p_{16}\\
                      0 & 0 & 0 & 0 & p_{25} & p_{26}\\
                      0 & 0 & 0 & 0 & p_{35} & p_{36}\\
                      0 & 0 & 0 & 0 & p_{45} & p_{46}\\
                      0 & 0 & 0 & 0 & 0 & 0 \\
                      0 & 0 & 0 & 0 & 0 & 0
                    \end{array}
 \right].
$$
We modify $\psi$ by the operations in Section \ref{opsection}. First apply $R_{6,5}(u)$
to $\psi$ for a new variable $u$. If $p_{46}+up_{45}\neq 0$, apply
$D_{5}(1/(p_{46}+up_{45}))$ and then
$R_{5,6}(-p_{45})$ to obtain a matrix of the form
$$\left[ \begin{array}{cccccc}
                      0 & p_{12} & p_{13} & p_{14} & * & * \\
                      0 & 0 & 0 & 0 & m^{24}_{56} & p_{26}+up_{25}\\
                      0 & 0 & 0 & 0 & m^{34}_{56} & p_{36}+up_{35}\\
                      0 & 0 & 0 & 0 & 0 & p_{46}+up_{45}\\
                      0 & 0 & 0 & 0 & 0 & 0 \\
                      0 & 0 & 0 & 0 & 0 & 0
                    \end{array}
                    \right].
$$
Hence by selecting a correct value for $u$ we would be done if
$V(m^{24}_{56},m^{34}_{56})\nsubseteq V(p_{45})$. We may assume
$$V(m^{24}_{56},m^{34}_{56})\subseteq V(p_{45}).$$ Similarly, we may also assume
$$V(m^{23}_{56},m^{34}_{56})\subseteq V(p_{35})
\mbox{ \ and \ }V(m^{23}_{56},m^{24}_{56})\subseteq V(p_{25}).$$ Therefore,
$$V(m^{23}_{56},m^{34}_{56},m^{24}_{56})\subseteq V(p_{25},p_{35},p_{45})=\emptyset.$$
Thus, $\{m^{23}_{56},m^{34}_{56},m^{24}_{56}\}$ is a regular sequence in
$k[x,y,z]$. If $p_{45}$ and $p_{46}$ are not relatively prime, there exists $\gamma $ such that
$m^{23}_{56}(\gamma)=0$, and $p_{45}(\gamma)=p_{46}(\gamma)=0$. Hence,
we may assume $p_{45}$ and $p_{46}$ are relatively prime, which leads a contradiction as we have
$$p_{45}m^{23}_{56}+p_{25}m^{34}_{56}-p_{35}m^{24}_{56}=0.$$
{\bf (vii)} If $\sigma \leq (1,6)(2,3)(4,5)$ then $r<2$. \\
To prove this case, consider
$$\psi (x:y)=\left[ \begin{array}{cccccc}
                      0 & 0 & p_{13} & p_{14} & p_{15} & p_{16}\\
                      0 & 0 & p_{23} & p_{24} & p_{25} & p_{26}\\
                      0 & 0 & 0 & 0 & p_{35} & p_{36}\\
                      0 & 0 & 0 & 0 & p_{45} & p_{46}\\
                      0 & 0 & 0 & 0 & 0 & 0 \\
                      0 & 0 & 0 & 0 & 0 & 0
                    \end{array}
 \right].
$$
Again by applying $R_{i,j}(u)$ and $D_{i}(v)$ for some $u,v$ we may assume that
$$V(m^{124}_{356})\subseteq V(p_{13},p_{23}) \mbox{ \ and \ }
V(m^{123}_{456})\subseteq
V(p_{14},p_{24}).$$ Hence $\{m^{124}_{356},m^{123}_{456}\}$ must be a regular sequence
 in $k[x,y]$. However this is impossible because the determinant of
$$\left[ \begin{array}{ccc}
                     \gcd (p_{13}, p_{14}) & p_{15} & p_{16} \\
                      \gcd (p_{23}, p_{24}) & p_{25} & p_{26} \\
                           0 & \gcd ( p_{35}, p_{45} ) & \gcd (p_{36},p_{46})
                    \end{array}
 \right]
$$
divides both $m^{124}_{356}$ and $m^{123}_{456}$. \\
{\bf (viii)} If $\sigma \leq (1,6)(2,5)(3,4)$ then $r<2$. \\
It is enough to consider a root of $m^{123}_{456}$ to prove this case.
\end{proof}
Note that in the above proof the last two cases prove Conjecture
\ref{conj2afternotations}
when $N\leq 6$ and $r\leq 2$. In the rest of
the paper we will generalize these ideas to prove the
conjecture for $r\leq 2$. To do this we examine the dimensions of these
varieties.
\subsection{Orbit dimensions and proof of second main result}
We now introduce notation for dimensions of these
varieties. In this subsection, for $\sigma \in \mathbf{P}(N)$, if the rank of
$\sigma$ is $s$, then we obtain two sequences of numbers $i_1,\ldots, i_s$ and
$j_1,\ldots, j_s$
 satisfying the following:
$$\sigma=(i_1,j_1)(i_2,j_2)\ldots(i_s,j_s)$$
 with $i_1<i_2<\dots<i_s$ and $i_a< j_a$ for all $1\leq a\leq s$.
In \cite{Melnikov1}, Melnikov gives a formula for the dimension of a
Borel orbit $B_{\sigma}$ for $\sigma $ in $\mathbf{P}(N)$ as follows:
\begin{itemize}
 \item $\displaystyle f_t(\sigma ):=\#\{j_p\mid p< t, j_p< j_t\} \,\,+\,\,
\#\{j_p\mid p<t, j_p< i_t\}$ for $2\leq t\leq s$,
 \item $\displaystyle\dim (B_\sigma)=Ns+\sum_{t=1}^s(i_t-j_t)-\sum_{t=2}^s f_t(\sigma)$.
\end{itemize}
We define a new subset of $\mathbf{RP}(N)$:
\begin{itemize}
\item $\mathbf{DP}(N)$ is the set of all $\sigma $ in $\mathbf{RP}(N)$ such that $\dim
(B_{{\sigma}'})=\dim (B_\sigma)-1$ whenever ${{\sigma}'}$ is a permutation
obtained by applying a single move of type {\RN{1}} to $\sigma$.
\end{itemize}
For instance, the following is the Hasse diagram of $\mathbf{DP}(8)$.
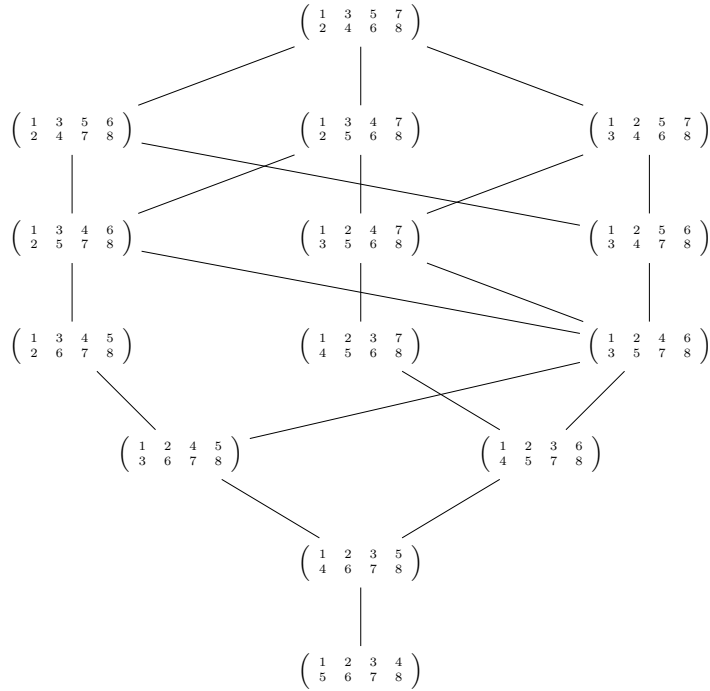
\begin{figure}[H]
    \centering
    \resizebox{0.6\textwidth}{!}{%
\begin{tikzpicture}[scale=.7]
  \node (yuzbes) at (0,6) {$\tiny{\left(
\begin{array}{cccc}
  1 & 3 & 5 & 7 \\2 & 4 & 6 & 8\end{array}\right)}$};
  \node (doksandokuz) at (0,3) {$\tiny{\left(
\begin{array}{cccc}
  1 & 3 & 4 & 7 \\2 & 5 & 6 & 8\end{array}\right)}$};
  \node (yuzuc) at (-8,3) {$\tiny{\left(
\begin{array}{cccc}
  1 & 3 & 5 & 6 \\2 & 4 & 7 & 8\end{array}\right)}$};
  \node (yetmisbes) at (8,3) {$\tiny{\left(
\begin{array}{cccc}
  1 & 2 & 5 & 7 \\3 & 4 & 6 & 8\end{array}\right)}$};
  \node (altmisdokuz) at (0,0) {$\tiny {\left(
\begin{array}{cccc}
  1 & 2 & 4 & 7 \\3 & 5 & 6 & 8\end{array}\right)}$};
  \node (doksanyedi) at (-8,0) {$\tiny{\left(
\begin{array}{cccc}
  1 & 3 & 4 & 6 \\2 & 5 & 7 & 8\end{array}\right)}$};
  \node (yetmisuc) at (8,0) {$\tiny {\left(
\begin{array}{cccc}
  1 & 2 & 5 & 6 \\3 & 4 & 7 & 8\end{array}\right)}$};
  \node (otuzuc) at (0,-3) {$\tiny {\left(\begin{array}{cccc}
  1 & 2 & 3 & 7 \\4 & 5 & 6 & 8\end{array}\right)}$};
  \node (doksanbir) at (-8,-3) {$\tiny {\left(\begin{array}{cccc}
  1 & 3 & 4 & 5 \\2 & 6 & 7 & 8\end{array}\right)}$};
  \node (altmisyedi) at (8,-3) {$\tiny {\left(\begin{array}{cccc}
  1 & 2 & 4 & 6 \\3 & 5 & 7 & 8\end{array}\right)}$};
  \node (altmisbir) at (-5,-6) {$\tiny {\left(\begin{array}{cccc}
  1 & 2 & 4 & 5 \\3 & 6 & 7 & 8\end{array}\right)}$};
  \node (otuzbir) at (5,-6) {$\tiny {\left(\begin{array}{cccc}
  1 & 2 & 3 & 6 \\4 & 5 & 7 & 8\end{array}\right)}$};
  \node (yirmibes) at (0,-9) {$\tiny {\left(\begin{array}{cccc}
  1 & 2 & 3 & 5 \\4 & 6 & 7 & 8\end{array}\right)}$};
  \node (bir) at (0,-12) {$\tiny {\left(\begin{array}{cccc}
  1 & 2 & 3 & 4 \\5 & 6 & 7 & 8\end{array}\right)}$};
  \draw (yuzbes) -- (doksandokuz) -- (doksanyedi) -- (altmisyedi) -- (altmisbir);
  \draw (yuzbes) -- (yuzuc) -- (doksanyedi) -- (doksanbir) -- (altmisbir) --
(yirmibes) -- (bir);
 \draw (yuzbes) -- (yetmisbes) -- (yetmisuc) -- (altmisyedi) -- (otuzbir) --
(yirmibes);
 \draw (yetmisbes) -- (altmisdokuz) -- (altmisyedi);
 \draw (yuzuc) -- (yetmisuc);
 \draw (doksandokuz) -- (altmisdokuz) -- (otuzuc) -- (otuzbir);
\end{tikzpicture}
 }%
    \caption{Hasse diagram of $\mathbf{DP}(8)$}
    \label{fig:3}
  \end{figure}
  Note that in the Hasse diagram of $\mathbf{DP}(8)$ all moves are of type \RN{5}.
This is generally the case, which we will prove below.
Before we do so, we will prove an easier result that will introduce the notation and argument style that will be necessary.

Fix $\sigma \in \mathbf{DP}(N)$. We use the our convention for $\sigma$ at the beginning of this subsection which implies $i_1=1$.

For $q\in \{1,\ldots,n\}$, let ${\sigma}'$ be
 the result of applying the move of type {\RN{1}}
 that deletes the $q$th transposition of $\sigma$, so that
 $${{\sigma}'=(1,j_1)\ldots(i_{q-1},j_{q-1})(i_{q+1},j_{q+1})\ldots (i_n,j_n)}.$$
As a matrix,
\[
   {\sigma}'=
 \left(\begin{array}{cccccccc}
  1    &  i_2 & \ldots &i_{q-1}& \widehat{i_q} &i_{q+1}  & \ldots & i_n \\
  \\
  j_1  &  j_2 & \ldots &j_{q-1}& \widehat{j_q} &j_{q+1}  & \ldots & j_n
     \end{array}\right).
\]

 Then by Melnikov's formula we have
  \begin{displaymath}
\dim (B_{\sigma})=Nn+\sum_{t=1}^n(i_t-j_t)-\sum_{t=2}^n f_t(\sigma)\, , \mbox { and}
\end{displaymath}
\begin{displaymath}
\dim (B_{{\sigma}'})=N(n-1)+\sum_{t=1}^{n-1}(i_t-j_t)-\sum_{t=2}^{n-1} f_t({\sigma}').
\end{displaymath}
 To simplify our calculation, we write $f_t(\sigma)=f_t^1(\sigma)+f_t^2(\sigma)$ for $2\leq t\leq n$,
where
\begin{displaymath}
f_t^1(\sigma)=\#\{j_p\mid p< t, j_p< j_t\} \mbox{ and }
 f_{t}^2(\sigma)=\#\{j_p\mid p< t, j_p< i_t\},
\end{displaymath}
and we use the notation:
\[
f_{t,q}^{l}({\sigma}')=
\left\{
\begin{array}{ll}
  { f_{t}^{l}({\sigma}')} & \mbox{ if } t\leq q-1 \\
   0                    &  \mbox{ if } t= q \\
   f_{t-1}^l({\sigma}') & \mbox{ if } t\geq q+1
\end{array}
\right. \]
for $l=1,2$.
\begin{lemma}\label{lemmaforpropnomove3}
 If $N\neq 2$ and the transposition $(1,\, N)$ appears in $\sigma$, then ${\sigma\notin \mathbf{DP}(N)}$.
\end{lemma}
\begin{proof}
If $(1,\, N)$ appears in $\sigma$ and $\sigma\in \mathbf{DP}(N)$,
let $q=2$ and ${\sigma}'$ be the result of deleting the second transposition from $\sigma$.
Since the $i$'s are increasing and $i_a< j_a$ for all $1\leq a\leq n$, we have $i_2=2$, so
\[
   {\sigma}'=
 \left(\begin{array}{ccccc}
  1    &  \widehat{2} &i_{3}&  \ldots & i_n \\
  \\
  N  &  \widehat{j_2} & j_{3}&  \ldots & j_n
     \end{array}\right).
\]
 Thus $3\leq {j_2} \leq N-1$. Let $j_2=N-b$ for some $1\leq b\leq N-3$. Then
\[
   {\sigma}=
 \left(\begin{array}{cccc}
  1    & 2 &  \ldots & i_n \\
  \\
 N  &  N-b &  \ldots & j_n
     \end{array}\right),
\]
and any number between $N-b$ and $N$ has to appear as a $j$ or an $i$ that is bigger than $j_2$. Therefore,
$$\left(\sum_{t=2}^n f_{t}^{1}({\sigma})-f_{t,2}^{1}({\sigma}')\right)+\left(\sum_{t=2}^n f_{t}^{2}({\sigma})-f_{t,2}^{2}({\sigma}')\right)=b-1.$$
Hence, $\dim(B_{\sigma})-\dim (B_{{\sigma}'})=N+2-N+b-(b-1)=3$, so $\sigma \notin \mathbf{DP}(N)$.
\end{proof}
Now we prove the main proposition of this subsection.
\begin{proposition}\label{propDPN}
If $\sigma \in \mathbf{DP}(N)$, then $j_p< j_t$ for all $p<t$, and therefore we cannot apply a move of type {\RN{3}} to $\sigma$. Conversely, if $\sigma\in \mathbf{RP}(N)$ and
we cannot apply a move of type of {\RN{3}} to $\sigma$, then $\sigma \in \mathbf{DP}(N)$.
\end{proposition}
\begin{proof}
Assume that $\sigma\in \mathbf{DP}(N)$. We will prove the following statement by induction on $k$:
\begin{equation}\label{st1}
j_{n-k}<\ldots <j_{n-1}<j_{n} \mbox{ and } \forall \,\, (p\, < n-k),\,\, j_p<j_{n-k}.\tag{$\ast$}
\end{equation}
Suppose $k=0$. To prove (\ref{st1}), we need to show that $\forall \,\, p< n$, $j_p<j_{n}$.
Let ${\sigma}'$ be obtained by deleting $n$th transposition of $\sigma$.
\[
   {\sigma}'=
 \left(\begin{array}{ccccc}
  1    &  i_2 & \ldots & i_{n-1}& \widehat{i_n} \\
  \\
  j_1  &  j_2 & \ldots &j_{n-1} & \widehat{j_n}
     \end{array}\right).
\]
 Since $\sigma \in \mathbf{DP}(N)$, we have
\begin{equation}\label{eq1}
1=\dim (B_{\sigma})-\dim (B_{{\sigma}'})=N+(i_n-j_n)-\bigg( f_n^1(\sigma)+ f_n^2({\sigma}) \bigg).
\end{equation}
Since the total number of possible $j$ except $j_n$ is $n-1$, and any number between $i_n$ and $j_n$ has to appear as $j$,
we have ${f_n^2({\sigma})=n-1-(j_n-i_n-1)}$. By the {equation (\ref{eq1})}, $f_n^1(\sigma)=n-1$, so that $\forall \,\, p< n$, $j_p<j_{n}$ is true. Therefore $j_n=N$.

Now assume the statement (\ref{st1}) is true for $k$. Then we can visualise $\sigma$ as follows:
\[
\sigma=
 \left(\begin{array}{cccccccc}
 1\enspace <&i_2\enspace <&\ldots\enspace <&i_{n-k-1}\enspace < \enspace{ i_{n-k}\enspace < \enspace i_{n-k+1}}\enspace < \enspace&\ldots& < i_n \\
 \\
 j_1\quad   &j_2\quad     &\ldots \quad    &j_{n-k-1}\enspace \enspace \,\,\,\,  \enspace{ j_{n-k}\enspace < \enspace j_{n-k+1}}\enspace < \enspace&\ldots& < j_n
     \end{array}\right)
\]
     We need to prove (\ref{st1}) for $k+1$, that is,
\[
j_{n-k-1}<\ldots <j_{n-1}<j_{n} \mbox{ and } \forall \,\, (p\, < n-k-1),\,\, j_p<j_{n-k-1}.
\]
By the second part of the inductive hypothesis for $k$, we have ${j_{n-k-1}<j_{n-k}}$ so
 the first part of (\ref{st1}) is already true, and we only need to show that the second part holds.
In other words, it is enough to show that $f_{n-k-1}^1(\sigma)=n-k-2$.
Let ${\sigma}'$ be obtained by deleting $(n-k-1)$th transposition of $\sigma$. Then we have
$$\sum_{t=n-k}^n f_{t}^{1}({\sigma})-f_{t,n-k-1}^{1}({\sigma}')=k+1.$$
Let $w:=\#\{i_p \mid i_{n-k-1}<i_p<j_{n-k-1}\}$. Then
\begin{align*}
f_{n-k-1}^2(\sigma)  =n-(k+2)-(j_{n-k-1}-i_{n-k-1}-1-w),
\end{align*}
and
$$\sum_{t=n-k}^n f_{t}^{2}({\sigma})-f_{t,n-k-1}^{2}({\sigma}')=k+1-w.$$

 By the fact that $\dim (B_{\sigma})-\dim (B_{{\sigma}'})=1$, we have
$f_{n-k-1}^1(\sigma)=n-k-2$. Thus the first claim is proved.

\indent Conversely, given $\sigma\in \mathbf{RP}(N)$, suppose that ${\sigma}'$ is the result of applying the move of type {\RN{1}}
 that deletes the $q$-th transposition of $\sigma$. Note that $f_{q}^1(\sigma)=q-1$ and
  ${\sum\limits_{t=q+1}^n f_{t}^1(\sigma)-f_{t,q}^{1}({\sigma}')=n-q}$. Hence,
\[
\sum_{t=2}^n f_{t}^1(\sigma)- f_{t,q}^{1}({\sigma}')=n-1.
\]
 Then,
\begin{displaymath}
\dim (B_{\sigma})-\dim(B_{{\sigma}'})=N+(i_q-j_q)-(n-1)-\bigg(\sum_{t=2}^n f_{t}^2(\sigma)-f_{t,q}^{2}({\sigma}')\bigg).
\end{displaymath}
We also have the difference $f_{t}^2(\sigma)-f_{t}^2({\sigma}')=0$ when $\displaystyle t\in \{1,\ldots, q-1\}$. Therefore,
\begin{align*}
\sum_{t=2}^n f_{t}^2(\sigma)- f_{t,q}^{2}({\sigma}')&= \sum_{t=q}^n \#\{j_p\mid p< t, j_p< i_t\}\\
& \,\, \quad  -\sum_{t=q+1}^n \#\{j_p\mid p< t,p\neq q, j_p< i_t\} \\
 & =\#\{j_p\mid j_p< i_q\}+\#\{i_t\mid  j_q< i_t\}.
\end{align*}
Let $F=\#\{j_p\mid j_p< i_q\}$ and $G=\#\{i_t\mid  j_q< i_t\}$.
Note that numbers between $i_q$ and $j_q$ must appear as $j_l$ for $l<q$ or as $i_s$ where $s>q$. Let
$a=\#\{j_p\mid  i_q<j_p<j_q\}$ and $b=\#\{i_t\mid   i_q< i_t<j_q\}$. Then $a+b=j_q-i_q-1$.
Let $A=\#\{j_p\mid  i_q<j_p\}$ and $B=\#\{i_t\mid   i_t<j_q\}$.
We have $A-a+B-b-1=n$. Therefore, $A+B=n+j_q-i_q$.
\begin{displaymath}
 \sigma =
 \left(\begin{array}{c}
 {{ \overbrace{
  {i_1,\ldots \ldots, i_q}
   \underbrace{\ldots }_{\mbox{b}}
  }^{\mbox{B}}}\, {\overbrace{ \ldots , i_n }^{\mbox{G}}}} \\
  { \underbrace{ {j_1,\ldots}}_{\mbox{F}}}\,{ \underbrace{ { \overbrace{\ldots}^{\mbox{a}}}{ j_q,\ldots \ldots, j_n}}_{\mbox{A}} }
  \end{array}
 \right)
\end{displaymath}
 Since $\sigma\in \mathbf{RP}(N)$, $A+B+F+G=2n$. Then $F+G=n-j_q+i_q$, and
$\dim (B_{\sigma})-\dim(B_{{\sigma}'})=1$, so $\sigma \in \mathbf{DP}(N)$.
\end{proof}

\begin{lemma} \label{lemmaA}
For every $X$ in $V(d)-L(d)$ we have $$r_{ij}(X)\geq j-i+1-n.$$
\end{lemma}
\begin{proof}
The rank of $X$ is $n$, so $r_{1N}(X)=n$. The result follows from the inequality
$$r_{ij}(X)+(i-1)+(N-j)\geq r_{1N}(X).$$
\end{proof}

We now define our last set of permutations:
\begin{itemize}
\item $\mathbf{MP}(r,N)$ is the set of minimal permutations in $\mathbf{P}(N)$
that appear as a
permutation in the form $\sigma _{\psi }$ for some $d$ and nonconstant morphism
$\psi: \mathbb{P}^{r-1}_k\rightarrow V(d)-L(d)$.
\end{itemize}
We now state and prove our second main result.
\begin{theorem} \label{mainthm2}
Conjecture \ref{conj2afternotations} holds for $r\leq 2$.
\end{theorem}
\begin{proof}
We have $\mathbf{MP}(1,N)= \{(1,n+1)(2,n+2)\dots (n,N)\}$ (see Example \ref{Ex1}). This means Conjecture
\ref{conj2afternotations} holds for $r=1$, because $N-n+(n+1)-1=N\leq N$. Hence it is enough to prove Conjecture
\ref{conj2afternotations} for $r=2$.

 Suppose that Conjecture
\ref{conj2afternotations} does not hold for $r= 2$. Then there exists
an $N$-tuple of nonincreasing integers $d=(d_1,d_2,\dots ,d_N)$, two
positive integers $\RowNumber$,
$\ColNumber$, and a nonconstant morphism $\psi:
\mathbb{P}^{1}_k\rightarrow {V(d)_{\RowNumber  \ColNumber}}-L(d)$
such that $N<2(\RowNumber  +\ColNumber
)$, or equivalently $n< \RowNumber  +\ColNumber$. Write
$\sigma _{\psi }=(i_1,j_1)(i_2,j_2)\ldots(i_n,j_n)$ with
$1=i_1<i_2<\dots<i_n$ and
$i_a< j_a$ for all $1\leq a\leq n$.

First assume that $\sigma \in \mathbf{DP}(N)$. By Proposition
\ref{propDPN},
we have $j_1<j_2<\dots<j_n=N$. Therefore, $\ColNumber   = j_1 - 1$ and
$\RowNumber   = N - i_n$.
Moreover, for every $a$ we
have $j_a>\ColNumber  $ and $i_a<N-\RowNumber  +1$.
Set $I:=\{i_1,\ldots, i_n\}$. Note that $\{1,\ldots,
\ColNumber\}\subseteq I$ since $\forall a, j_a>\ColNumber $. So, $i_a=a$
if $a\leq \ColNumber $. Similarly, $j_{a-n}=a$ if $a\geq N-\RowNumber
+1$.
Set $$\underline{i}:=n-\RowNumber  +1 \mbox{ \ and \  }
\overline{i}:=\ColNumber $$
and
     $$\underline{j}:=N-\RowNumber  +1 \mbox{ \ and \
}\overline{j}:=\ColNumber  +n.$$
We have
$$\overline{i}-\underline{i}=\RowNumber  +\ColNumber
-n-1=\overline{j}-\underline{j}.$$
In particular, $\underline{i}\leq \overline{i}$ and $\underline{j}\leq \overline{j}$,
since we assumed that $n< \RowNumber  +\ColNumber$.

For $i\leq i'$ and $j\leq j'$, let $A^{i,i'}_{j,j'}$ denote the
submatrix of the partial permutation matrix $P_{\sigma}$ obtained by
considering the rows from  $i$ to  $i'$ and columns from  $j$ to  $j'$.
Note that $A^{1,\overline{i}}_{1,N}$ has $\ColNumber $ many $1$'s and
$A^{1,N}_{\underline{j},N}$ has $\RowNumber $ many $1$'s. Hence
$A^{1,\overline{i}}_{\underline{j},N}$ must have
$(\RowNumber+\ColNumber-n)$ many $1$'s. However, there is no $1$ in
$A^{1,\underline{i}-1}_{\underline{j},N}$; otherwise there would exist
$a$ such that $i_a<\underline{i}=n-\RowNumber +1$ and $j_a\geq
\underline{j}$, which leads to a contradiction by considering the number of
$1$'s in the region determined by the union of $A^{1,i_a-1}_{1,N}$ and
$A^{1,N}_{j_a+1,N}$. Similarly, there is no $1$ in
$A^{1,\overline{i}}_{\overline{j}+1,N}$.
Hence the $(\RowNumber  +\ColNumber -n)\times (\RowNumber
+\ColNumber-n) $-submatrix
$A^{\underline{i},\overline{i}}_{\underline{j},\overline{j}}$ contains
$(\RowNumber  +\ColNumber-n)$ many $1$'s.
Thus, $A^{\underline{i},\overline{i}}_{\underline{j},\overline{j}}$ is
the identity matrix of dimension $\RowNumber  +\ColNumber -n$.
\begin{displaymath}
\begin{blockarray}{ccccccccccccccccc}
 &1 & & & & \ColNumber & & & & & \underline{j}& &\overline{j} & & & \\
\begin{block}{[ cccccccccccccccc]l}
      & \tikzmark{left1}0 &  &\ldots & & 0 & & & & & & & & & & & \\
      &  &  & & &   & & & & & & & & & & &\\
&\vdots & & & &\vdots & &  &  & &\tikzmark{left2} & & &&&& \underline{i}=n-\RowNumber + 1 \\
&   &  & & &  & & & & &  &  A^{\underline{i},\overline{i}}_{\underline{j},\overline{j}}& & & & &\\
& 0 &  &\ldots & & 0 \tikzmark{right1} & & & & & & &\tikzmark{right2} & & && \overline{i}= \ColNumber  \\
&   & & & &  & & & & & & & & & & & \\
&   & & & &  & & & & & & & & & & & \\
&   & & & &  & & & & & & & & & & & \\
&   & & & &  & & & & & & & & & & & \\
&   & & & &  & & & & &\tikzmark{left3}0 & &\ldots &  & 0 & &{\underline{j}}=N-\RowNumber+1\\
&   & & & &  & & & & &  & & & &  & & \\
&   & & & &  & & & & & \vdots & & & &\vdots & & \\
&   & & & &  & & & & &  & & & & & &\\
&   & & & &  & & & & &  & & & & & &\\
&   & & & &  & & & & & 0& & \ldots& & 0\tikzmark{right3}& & N \\
\end{block}
\tikzdrawbox{1}{thick,black}
\tikzdrawbox{2}{thick,black}
\tikzdrawbox{3}{thick,black}
\end{blockarray}
\end{displaymath}
 This in
particular means that
for every $X$ in the image of  $\psi $ we have
$$r_{\underline{i},\underline{j}-1}(X)=r_{n-\RowNumber  +1,N-\RowNumber
}(X)=0$$ and
$$r_{\overline{i}+1,\overline{j}}(X)=r_{\ColNumber  +1,\ColNumber
+n}(X)=0.$$
Since $k$ is algebraically closed, there exists a root of the minor
$m^{\underline{i},\underline{i}+1\ldots ,
\overline{i}}_{\underline{j},\underline{j}+1\ldots ,\overline{j}}$.
Thus, there exists $X$ in the image of  $\psi $ such that
$$r_{\underline{i},\overline{j}}(X)\leq \overline{i}-\underline{i},$$
which means
$$r_{n-\RowNumber  +1,\ColNumber  +n}(X)\leq  \RowNumber  +\ColNumber
-n-1.$$
Lemma \ref{lemmaA}
implies that  for every $X$ in the image of
$\psi $ we have
$$r_{n-\RowNumber  +1,\ColNumber  +n}(X)\geq \ColNumber
+n-(n-\RowNumber  +1)+1-n=\RowNumber  +\ColNumber  -n.$$
This is a contradiction, so we are done with the case $\sigma \in \mathbf{DP}(N)$.

Now assume that $\sigma \notin \mathbf{DP}(N)$. We recursively define
perturbations of $\psi $ so that we can again use the square submatrix
$A^{\underline{i},\overline{i}}_{\underline{j},\overline{j}}$ to get a contradiction
similar to that of the previous case. Set $\psi ^0=\psi $, $n_0=0$, $Z_0=\emptyset $. We
have a rational map $\psi ^0: \mathbb{A}_k^{2+n_0}-Z_0\rightarrow V_N$. Now given
$$\psi ^s: \mathbb{A}_k^{2+n_s}-Z_s\rightarrow V_N$$ we define $\psi ^{s+1}:
\mathbb{A}_k^{2+n_{s+1}}-Z_{s+1}\rightarrow V_N$ when $\sigma_s=\sigma_{\psi^s}$ is
not in $\mathbf{DP}(N)$.

Assume $\sigma_s\notin \mathbf{DP}(N)$. By Proposition \ref{propDPN},
there exists a  move of type {\RN{3}}
that we can apply to $\sigma_s$. Hence, we may define
$$l'_s:=\min\{\,l'\,|\,l'<l<\sigma_s(l)<\sigma_s(l')\,\}$$
and
$$l_s:=\sigma_s
\left(\min\{\,\sigma_s(l)\,|\,l'_s<l<\sigma_s(l)<\sigma_s(l'_s)\,\}\right).$$
In case $l'_s< \underline{i}$, we define
$$n_{s+1}:=n_s+l_s-l'_s+1.$$
Note that $l_s>l'_s$ and so $n_{s+1}\geq n_s+2$. Hence the affine variety $Z_s$ can
be considered as a subvariety of $\mathbb{A}_k^{2+n_{s+1}}$ by considering
$\mathbb{A}_k^{2+n_{s}}\subset \mathbb{A}_k^{2+n_{s+1}} $.
Here we write $(x,y,u_1,u_2,\ldots u_{n_{s+1}})$ to denote a point in
$\mathbb{A}_k^{2+n_{s+1}}$.
Hence $\mathbb{A}_k^{2+n_{s}}$ corresponds to the points where
${u_{n_s+1}=\ldots=u_{n_{s+1}}=0}$.

Represent $\psi^s
$ by a matrix $(\psi^s_{ij})$ whose $(i,j)$-entry $\psi^s_{i,j}\in k(x,y, u_1,u_2,\ldots u_{n_{s}})$ is
a rational function. Let $p_i$ denote the
entry $\psi^s_ {i,\sigma(l_s)}$ which is a polynomial. Also let $\bar{p}$ be the greatest common divisor of
$p_{l'_s},p_{l'_s+1},\ldots , p_{l_s} $. Set $p'_i:=p_i/\bar{p}$
for $i\in \{l'_s,l'_s+1,\ldots,l_s\} $. We define
$$Z_{s+1}=Z_s\cup V\left(\, u_{n_s+1}\, , \,
\sum_{i:=l'_s}^{l_s}u_{n_s+i-l'_s+1}p'_i\, \right).$$
We obtain $\psi ^{s+1}$ from $\psi ^s$ by first applying $D_{l'_s}(u_{n_s+1})$, then
$R_{i,l'_s}(u_i)$ where $l'_s<i\leq l_s$, then
$D_{i}(\sum_{i:=l'_s}^{l_s}u_{n_s+i-l'_s+1}p'_i)$ for  $l'_s<i\leq l_s$, and
finally applying $R_{l'_s,i}(-p'_{i})$ for  $l'_s<i\leq l_s$.
 Notice that $p'_i$ also depends on $s$, so we write
$p'_{s,i}$ instead of $p'_i$ when $s$ is not clear.

We can repeat this process until it is no longer possible to find a move of type
{\RN{3}}
with $l'_s< \underline{i}$. At the end of this part of the process we obtain a rational map $\psi^t$ for some $t$.
Then we can continue with the symmetric (with
respect to the diagonal of the matrix running from the lower left entry to the upper
right entry) operations assuming
 $\psi^s$ is defined for $s\geq t$. We define $\psi^{s+1}$ as follows:
$$l'_s:=\max\{\,l'\,|\,l'<l<\sigma_s(l)<\sigma_s(l')\,\}$$
and
$$l_s:= \sigma_s\left(\max\{\,{\sigma}_s(l)\,|\,l'_s<l<\sigma_s(l)<\sigma_s(l'_s)\,\}\right).$$
We repeat the symmetric operations as long as we have $\sigma_s(l'_s)> \overline{j}$.
 We define $n_{s+1}:=n_s+\sigma_s(l'_s)-\sigma_s(l_s)+1$.
 Let $p_j$ denote the entry $\psi_{l_s,j}^s$ which is a polynomial. Also let $\bar{q}$ be the greatest  common divisor of $p_{\sigma_s(l_s)},p_{\sigma(l_s)+1},\ldots, p_{\sigma_s(l'_s)}$.
Similarly, we write $p'_{s,j}$ instead of $p'_j$ when $s$ is not clear.

At the end of this process we obtain a rational map ${\psi}^{\bar{t}}$ from the quasi affine
variety $U=\mathbb{A}_k^{2+n_{\bar{t}}}-Z_{\bar{t}}$ to $V_N$ where we have
$r_{\underline{i},\underline{j}-1}(X)=r_{n-\RowNumber  +1,N-\RowNumber
}(X)=0$ and
$r_{\overline{i}+1,\overline{j}}(X)=r_{\ColNumber  +1,\ColNumber
+n}(X)=0$
for every $X$ in the image of this rational map. Denote the composition of
$m^{\underline{i},\dots,\overline{i}}_{\underline{j},\dots
\overline{j}}$ with $\psi ^{\bar{t}}$ by $m$. Notice that $m$ is a polynomial in
$k[x,y,u_1,\ldots u_{n_{\bar{t}}}]$.
Define another polynomial $p$ in the same polynomial algebra as follows:
$$p=\left(\prod_{s=0}^{t-1} \, u_{n_s+1} \sum_{i=l'_s}^{l_s}u_{n_s+i-l'_s+1}p'_{s,i}\right)
\left(\prod_{s=t}^{\bar{t}-1} \, u_{n_s+1} \sum_{j=\sigma(l_s)}^{\sigma(l'_s)}u_{n_s+\sigma(l'_s)-j+1}p'_{s,j}\right).$$
For ${1\leq s\leq t}$, we have $p'_{s,l'_s}, p'_{s,l'_{s}+1},\ldots , p'_{s,l_s}$ are relatively prime. Similarly, for ${t+1\leq s\leq \bar{t}}$ we have $p'_{s,\sigma(l_s)}, p'_{s,\sigma(l_{s})+1},\ldots ,  p'_{s,\sigma(l'_s)}$ are relatively prime. Moreover, the polynomial $m$ has an irreducible factor
in $k[x,y]$ or for some $s$,  the polynomial $m$ has an irreducible factor in the form
$fu_{n_s}+g$, where $f$ and $g$ are in $k[x,y,u_1,\ldots u_{{n_s}-1}]$
so that $f$ is neither an associate of $p'_{s,{l_s}}$ nor
 $p'_{s,{\sigma(l_s)}}$.
 Hence, there exists a
solution to the equations $p=1$ and $m=0$, which is again a contradiction by Lemma
\ref{lemmaA}.
\end{proof}
\section{Examples and Problems}\label{Examples}
 The last inequality in Conjecture
 \ref{conj2afternotations} is equivalent to
 \begin{displaymath}
r\leq {\left\lfloor{ \log_2{\left( \frac{N}{\RowNumber
+\ColNumber}\right)}}\right\rfloor}+1.
\end{displaymath}
One might ask how strict this upper bound for $r$ is.

In all the following examples, we define $\psi$ from
$\mathbb{P}^{r-1}_k$ to $V(d)_{\RowNumber \ColNumber}-L(d)$ for
 different values of $r$, $N$ and $\RowNumber+\ColNumber$ where
  $r= {\left\lfloor{ \log_2{\left( \frac{N}{\RowNumber
+\ColNumber}\right)}}\right\rfloor}+1$.
It follows that we do not have a better upper bound for $r$ in these cases.
\begin{example}\label{Ex1}
For $r=1$, $N=2n$, $d=(0,\ldots,0)$ and $\RowNumber
+\ColNumber=N$, define
\begin{displaymath}
\psi(x)=
\left[
\begin{array}{ccc|ccc}
 & & & & & \\
 & 0 & & &M& \\
 & & & & & \\
\hline
 & & & & & \\
 & 0 & & &0& \\
 & & & & & \\
\end{array}
 \right]
\mbox{ }\mbox{ where}\mbox{ }M = \left[
\begin{array}{ccc}
x  &       & 0\\
   &\ddots &  \\
 0 &       & x\\
 \end{array}
  \right].
\end{displaymath}
Note that $\sigma_{\psi}=(1,n+1)(2,n+2)\dots (n,N)$. This example shows that
$$\mathbf{MP}(1,N)= \{(1,n+1)(2,n+2)\dots (n,N)\}.$$
\end{example}
\begin{example}\label{Ex2}
For $r=2$ and $N=4$, $d=(0,0,0,0)$ and $\RowNumber  +\ColNumber =2$, define
\begin{displaymath}
\psi(x,y)=
\left[
\begin{array}{rrrr}
0&x&y&0\\
0&0&0&y\\
0&0&0&-x\\
0&0&0&0\\
 \end{array}
 \right].
\end{displaymath}
In this example, $\sigma_{\psi}=(1,2)(3,4)$. Hence, $$\mathbf{MP}(2,4)=
\{(1,2)(3,4)\}.$$
\end{example}
\begin{example}\label{Ex3}
For $r=2$, $N=6$, $d=(0,-1,-1,-1,-1,-1)$, and $\RowNumber  +\ColNumber=3$, set:
\begin{displaymath}
\psi(x,y)=
\left[
\begin{array}{rrrrrr}
0&x^2&xy&y^2&0&0\\
0&0&0&0&y&0\\
0&0&0&0&-x&y\\
0&0&0&0&0&-x\\
0&0&0&0&0&0\\
0&0&0&0&0&0\\
 \end{array}
 \right].
\end{displaymath}
 Here, we have $\sigma_{\psi}=(1,2)(3,5)(4,6)$. Considering the
 Hasse diagram for $\mathbf{RP}(6)$ in Figure \ref{fig:2} and symmetry it is clear
 that $$\mathbf{MP}(2,6)=\{ \,(1,2)(3,5)(4,6)\, , \, (1,3)(2,4)(5,6)\, \}.$$
\end{example}
The above example can be generalized:
\begin{example}
 For $r=2$, $N=2n$, $d=(0,-n+2,\ldots,-n+2)$, and  $\RowNumber
+\ColNumber=n$, set:
\begin{displaymath}
\psi(x,y)=
\left[
\begin{array}{rrrrrrrrr}
0&x^{n-1}& x^{n-2}y & \ldots & y^{n-1}& 0 &   & \ldots & 0\\
 &   0   &  0       &\ldots  & 0      & y &   &        &\vdots\\
 &       &  0       &        &        & -x & y &        & \\
 &       &          &        &        &   & -x & \ddots & 0 \\
 &       &          &        &   \ddots      &   &   & \ddots & y\\
 &       &          &        &        &  &   &        & -x \\
 &       &          &        &        &   &   &        & 0 \\
 &       &          &        &        &   &   &        & \vdots\\
 &       &          &        &        &   &   &        & 0 \\
 \end{array}
 \right].
\end{displaymath}
\end{example}
We can use the above examples to obtain new ones by the chess board construction:
\begin{construction}
Let $(l_1,l_2,\ldots, l_m)$ be an $m$-tuple of positive integers and
$V(d)_{(l_1,l_2,\ldots, l_m)}$ the subvariety of $V(d)$
such that $x_{ij}=0$ when $l_1+l_2+\ldots+ l_{(s-1)}+1\leq i< j \leq l_1+l_2+\ldots+
l_{s}$ for some $1\leq s\leq m$.  For example, the following matrix
$\psi(x,y)$ is in $V(d)_{(1,3,2)}$ where $d$ is $6$-tuple of nonincreasing
integers:
\[\psi(x,y)=
\begin{tikzpicture}[baseline=-\the\dimexpr\fontdimen22\textfont2\relax ]
\matrix (m)[matrix of math nodes,left delimiter={[},right delimiter={]}]
{
0&x^2&xy&y^2&0&0\\
0&0&0&0&y&0\\
0&0&0&0&-x&y\\
0&0&0&0&0&-x\\
0&0&0&0&0&0\\
0&0&0&0&0&0\\
};

\begin{pgfonlayer}{myback}
\highlight[black]{m-1-1}{m-1-1}
\highlight[black]{m-2-2}{m-4-4}
\highlight[black]{m-5-5}{m-6-6}
\end{pgfonlayer}
\end{tikzpicture}.
\]
Take
$\psi_1 \in V(d_1)_{(l_1^1,l_2^1,\ldots,l_m^1)}$ where $d_1$ is $N_1$-tuple
nonpositive integers and ${\psi_2 \in V(d_2)_{(l_1^2,l_2^2,\ldots,l_m^2)}}$ where
$d_2$ is $N_2$-tuple of nonincreasing integers.
We arrange a $(N_1+N_2)\times (N_1+N_2)$ matrix in a $2m\times 2m$-chessboard as
follows:
The $ij$-square contains a
$l_{\left\lfloor{{\frac{i+1}{2}}}\right\rfloor}^{\varepsilon_i}\times
l_{\left\lfloor{{\frac{j+1}{2}}}\right\rfloor}^{\varepsilon_j}$ matrix such that
$\varepsilon_k=1$ if
$k$ is odd
or $\varepsilon_k=2$ if $k$ is even integer. Now we color the $ij$ square black if
$\varepsilon_i\neq \varepsilon_j$ and white if $\varepsilon_i= \varepsilon_j$.
Fill in the $ij$ square with zeros if it is a black square and otherwise fill it in with
$(x_{{i'}{j'}})$ where $\underline{i}\leq i'\leq \overline{i}$ and
$\underline{j}\leq j'\leq \overline{j}$
part of $\psi_{\varepsilon_i}$ where $\underline{i}, \overline{i}, \underline{j}, \overline{j}$ are defined by
\begin{displaymath}
\displaystyle{\underline{s}=\sum_{m=1}^{{\left\lfloor{{\frac{s+1}{2}}}\right\rfloor}-1}l_m^{\varepsilon_s}+1}
\mbox{ and }
\displaystyle{\overline{s}=\sum_{m=1}^{\left\lfloor{\frac{s+1}{2}}\right\rfloor}
l_m^{\varepsilon_s}}.
\end{displaymath}
\end{construction}
For instance, using chessboard construction we can obtain an example:
\begin{example}
 For $r=2$, $N=4+6=10$, $d=(0,0,-1,-1,-1,-1,-1,-1,-1,-1)$, and $\RowNumber
+\ColNumber=3+2=5$ we obtain an example by applying the chess board construction on
the morphisms in Examples \ref{Ex2} and \ref{Ex3}.
\begin{center}\chessboard{
\TextOnWhite{$0$}\TextOnBlack{}\TextOnWhite{${x^2}$}\TextOnWhite{${xy}$}\TextOnWhite{${y^2}$}%
\TextOnBlack{}\TextOnBlack{}\TextOnWhite{$0$}\TextOnWhite{$0$}\TextOnBlack{}\\
\TextOnBlack{}\TextOnWhite{$0$}\TextOnBlack{}\TextOnBlack{}\TextOnBlack{}%
\TextOnWhite{$ x$}\TextOnWhite{$y$}\TextOnBlack{}\TextOnBlack{}\TextOnWhite{$0$}\\
\TextOnWhite{$ 0$}\TextOnBlack{}\TextOnWhite{$ 0$}\TextOnWhite{$ 0$}\TextOnWhite{$
0$}%
\TextOnBlack{}\TextOnBlack{}\TextOnWhite{$\,\,y$}\TextOnWhite{$\,\,0$}\TextOnBlack{}\\
\TextOnWhite{$ 0$}\TextOnBlack{}\TextOnWhite{$ 0$}\TextOnWhite{$ 0$}\TextOnWhite{$
0$}%
\TextOnBlack{}\TextOnBlack{}\TextOnWhite{$-x$}\TextOnWhite{$\,\,y$}\TextOnBlack{}\\
\TextOnWhite{$ 0$}\TextOnBlack{}\TextOnWhite{$ 0$}\TextOnWhite{$ 0$}\TextOnWhite{$
0$}%
\TextOnBlack{}\TextOnBlack{}\TextOnWhite{$\,\,0$}\TextOnWhite{$-x$}\TextOnBlack{}\\
\TextOnBlack{}\TextOnWhite{$ 0$}\TextOnBlack{}\TextOnBlack{}\TextOnBlack{}%
\TextOnWhite{$ 0$}\TextOnWhite{$ 0$}\TextOnBlack{}\TextOnBlack{}\TextOnWhite{$\,\,y$}\\
\TextOnBlack{}\TextOnWhite{$ 0$}\TextOnBlack{}\TextOnBlack{}\TextOnBlack{}%
\TextOnWhite{$ 0$}\TextOnWhite{$ 0$}\TextOnBlack{}\TextOnBlack{}\TextOnWhite{$-x$}\\
\TextOnWhite{$0$}\TextOnBlack{}\TextOnWhite{$0$}\TextOnWhite{$0$}\TextOnWhite{$0$}%
\TextOnBlack{}\TextOnBlack{}\TextOnWhite{$0$}\TextOnWhite{$0$}\TextOnBlack{}\\
\TextOnWhite{$0$}\TextOnBlack{}\TextOnWhite{$0$}\TextOnWhite{$0$}\TextOnWhite{$0$}%
\TextOnBlack{}\TextOnBlack{}\TextOnWhite{$0$}\TextOnWhite{$0$}\TextOnBlack{}\\
\TextOnBlack{}\TextOnWhite{$0$}\TextOnBlack{}\TextOnBlack{}\TextOnBlack{}%
\TextOnWhite{$0$}\TextOnWhite{$0$}\TextOnBlack{}\TextOnBlack{}\TextOnWhite{$0$}
}
\end{center}
\end{example}

We also have other well-known constructions like the Koszul complex construction
\cite{Avramov} giving
us examples as below.
\begin{example}

For $r=3$, $N=8$ and $\RowNumber  +\ColNumber=2$, define
\begin{displaymath}
\psi(x,y,z)=
 \left[
\begin{array}{rrrrrrrr}
0& x    &  y   & z      &  0   & 0   & 0   & 0 \\
0& 0    &  0   & 0      &   y  &-z   & 0   & 0 \\
0& 0    &  0   & 0      &  -x  & 0   & z   & 0 \\
0& 0    &  0   & 0      &   0  & x   &-y   & 0\\
0& 0    &  0   & 0      &   0  & 0   & 0   & z\\
0& 0    &  0   & 0      &   0  & 0   &0    & y\\
0& 0    &  0   & 0      &   0  & 0   &0    & x\\
0& 0    &  0   & 0      &   0  & 0   &0    & 0
\end{array}
 \right].
 \end{displaymath}
 In this example, $\sigma_\psi=(1,2)(3,5)(4,6)(7,8)$.
\end{example}
We end with a few questions for future research.
 Notice that all examples discussed above are in $\mathbf{DP}(N)$.
 Hence one can ask :
\begin{question}
Is $\mathbf{MP}(r,N)\subseteq \mathbf{DP}(N) $ ?
\end{question}
For all these examples $\frac{N}{2^{r-1}}$ is an integer.
 For instance, we can find an example see Example \ref{Ex12} for $r=3$, $N=12$ but
we do not know the answer of the following question:
\begin{question}\label{question2}
 Is there any example for $r=3$ and $N=10$? More precisely, can we say that
$\mathbf{MP}(3,10)$ is nonempty?
\end{question}

 Note that the following example can not be obtained by the constructions we
mentioned
above.
\begin{example}\label{Ex12}
For $r=3$, $N=12$, $d=(0,0,-1,\ldots,-1)$ and $\RowNumber  +\ColNumber=3$, consider

\begin{displaymath}
 \left[
\begin{array}{rrrrrrrrrrrr}
0& x    &  y^2 & yz     &  z^2    & 0    &  0   & 0   &  0   & 0   & 0 & 0 \\
0& 0    &  0   & 0      &  0      & y^2  &  yz  & 0   &  z^2 & 0   & 0 & 0 \\
0& 0    &  0   & 0      &  0      & -x    &  0   & -z   &  0   & 0   & 0 & 0 \\
0& 0    &  0   & 0      &   0     & 0    &  -x   & y   &  0   & z   & 0 & 0 \\
0& 0    &  0   & 0      &   0     & 0    &  0   & 0   &  -x   &- y   & 0 & 0 \\
0& 0    &  0   & 0      &   0     & 0    &  0   & 0   &  0   & 0   & -z & 0 \\
0& 0    &  0   & 0      &   0     & 0    &  0   & 0   &  0   & 0   & y & -z \\
0& 0    &  0   & 0      &   0     & 0    &  0   & 0   &  0   & 0   & x & 0 \\
0& 0    &  0   & 0      &   0     & 0    &  0   & 0   &  0   & 0   & 0 & y \\
0& 0    &  0   & 0      &   0     & 0    &  0   & 0   &  0   & 0   & 0 & -x \\
0& 0    &  0   & 0      &   0     & 0    &  0   & 0   &  0   & 0   & 0 & 0\\
0& 0    &  0   & 0      &   0     & 0    &  0   & 0   &  0   & 0   & 0 & 0\\
\end{array}
\right].
\end{displaymath}
\end{example}

Note that we can obtain an example for $r=3$ and
$N=4s$
for every $s\geq 2$ by using the examples for $r=3$, $N=8$ and the example for
$r=3$, $N=12$ and applying the chessboard construction as many times as necessary.
If the answer to question \ref{question2} is negative, then one can ask the
following question:
\begin{question}
Do there exist any periodicity results about nonemptiness of $\mathbf{MP}(r,N)$?
\end{question}

Another observation we make about these examples is that there always exists a
sequence of permutations $\sigma _1<\sigma _2 <\dots <\sigma _r$ such that the image
of the morphism contains a point from each Borel orbit corresponding to the these
$\sigma _i$'s and each pair of consecutive ${\sigma_i}$'s consist of distinct
transpositions. For example,
putting $x=1$ and $y=0$ to $\psi$ in Example \ref{Ex2},we get a point in the Borel
orbit corresponding to permutation $\sigma_2=(1,2)(3,4)$.
and putting $x=0$ and $y=1$, we get $\sigma_1=(1,3)(2,4)$. Hence one could ask the
following question:

\begin{question}
 Given $\sigma $ in $\mathbf{MP}(r,N)$ does there always exists a morphism
 $\psi: \mathbb{P}^{r-1}_k\rightarrow V(d)-L(d)$
with a sequence permutations $\sigma _1<\sigma _2 <\dots <\sigma _r$ and points
$X_1$, $X_2$, $\ldots$ ,$X_r$ in the image of $\psi $ such that $\sigma
_{\psi}=\sigma $ and $X_i$ is in the Borel orbit of $\sigma _i$ for all $i$ and
$\sigma _i$ and $\sigma_{i+1}$ has no common transpositions?
\end{question}
 If the answer is affirmative to this question then one can say that the inequalities
 $$\frac{n(n+1)}{2}\leq \dim(B_{{\sigma}_i})\leq n^2$$
 and
 $$\dim (B_{{\sigma}_i})+ {\left\lceil{ \frac{n}{2}}\right\rceil} \leq \dim (B_{{\sigma}_{i+1}})$$
 hold and they give the inequality $N\geq 2r$.

 Note that Allday and Puppe \cite{AlldayPuppe} have related results: If $k$, $A$,
$r$, $N$, and $M$ are as in Conjecture
\ref{carlssonsconjecture}, then they prove $N\geq 2r$. Moreover, Avramov,
Buchweitz and Iyengar \cite{Avramov} verified that $N\geq 2r$ in a more general case.

\bibliographystyle{plain}
\bibliography{newreferenceR2}
\end{document}